\documentclass[11pt,reqno]{amsart}
\usepackage{amsmath,amsfonts,amssymb,amsthm,amscd,comment,euscript}
\usepackage[all]{xy}
\usepackage{graphicx}
\usepackage{mathptmx}
\usepackage{enumerate} %allows to change enumeration items easily
\usepackage[colorlinks=true]{hyperref} %to generate pdf with links
\usepackage[usenames,dvipsnames]{xcolor}
\usepackage{enumitem}

\xymatrixcolsep{1.9pc}                          % Adjust size of diagrams.
\xymatrixrowsep{1.9pc}
\newdir{ >}{{}*!/-5pt/\dir{>}}                  % Make better tailed arrows

 \calclayout

\swapnumbers
\theoremstyle{plain}
\newtheorem{lem}{Lemma}[section]
\newtheorem{cor}[lem]{Corollary}

\newtheorem{thm}[lem]{Theorem}

\theoremstyle{definition}
\newtheorem{ex}[lem]{Example}
\newtheorem{rem}[lem]{Remark}
\newtheorem{dfn}[lem]{Definition}

\newtheorem{ntt}[lem]{Notation}

\renewcommand{\phi}{\varphi}
\renewcommand{\leq}{\leqslant}
\renewcommand{\geq}{\geqslant}
\renewcommand{\epsilon}{\varepsilon}

\renewcommand{\kappa}{\varkappa}

\DeclareMathOperator{\spec}{Spec}

\DeclareMathOperator{\hocolim}{hocolim}

\DeclareMathOperator{\Hom}{Hom} 
 \DeclareMathOperator{\id}{id}

 \DeclareMathOperator{\Mor}{Mor}
 \DeclareMathOperator{\colim}{colim}

 \DeclareMathOperator{\kr}{Ker}
 
\DeclareMathOperator{\coker}{Coker} \DeclareMathOperator{\nis}{nis}

\newcommand{\lra}[1]{\bl{#1}\longrightarrow\relax}
\newcommand{\bl}[1]{\buildrel #1\over}
\newcommand{\cc}{\mathcal}
\newcommand{\bb}{\mathbb}

\newcommand{\Fr}{\cc Fr_0(k)}
\newcommand{\op}{{\textrm{\rm op}}}

\newcommand{\wh}{\widehat}

\newcommand{\Fun}{\mathrm{Fun}}

\newcommand{\gmp}{\bb G_m^{\wedge 1}}
\newcommand{\gmpn}{\bb G_m^{\wedge n}}
\newcommand{\uhom}{\underline{\Hom}}
\newcommand{\gmpnl}{\bb G_m^{\wedge {n+1}}}
\begin{document}

\footskip30pt

%\baselineskip=1.5\baselineskip

\title{The triangulated categories of framed bispectra and framed motives}
\author{Grigory Garkusha}
\address{Department of Mathematics, Swansea University, Fabian Way, Swansea SA1 8EN, UK}
\email{g.garkusha@swansea.ac.uk}

\author{Ivan Panin}
\address{St. Petersburg Branch of V. A. Steklov Mathematical Institute,
Fontanka 27, 191023 St. Petersburg, Russia}

\email{paniniv@gmail.com}

%\thanks{}

\begin{abstract}
An alternative approach to classical Morel--Voevodsky stable motivic homotopy
theory $SH(k)$ is suggested. The triangulated category of framed bispectra $SH_{\nis}^{fr}(k)$ and effective
framed bispectra $SH_{\nis}^{fr,eff}(k)$ are introduced in the paper. Both triangulated categories
only use Nisnevich local equivalences and have nothing to do with any kind of motivic equivalences.
It is shown that $SH_{\nis}^{fr}(k)$ and $SH_{\nis}^{fr,eff}(k)$ recover classical Morel--Voevodsky triangulated
categories of bispectra $SH(k)$ and effective bispectra $SH^{eff}(k)$ respectively.

We also recover $SH(k)$ and $SH^{eff}(k)$ as the triangulated category of framed motivic spectral
functors $SH_{S^1}^{fr}[\cc Fr_0(k)]$ and the triangulated category of framed motives
$\cc {SH}^{fr}(k)$ constructed in the paper.
\end{abstract}

\dedicatory{In memory of A. A. Suslin}



\maketitle

\thispagestyle{empty} \pagestyle{plain}

\newdir{ >}{{}*!/-6pt/@{>}} %this command is to define the arrow of the type \ar@{ >->} (spacing of arrows is needed sometimes)

%\tableofcontents

\section{Introduction}

Stable motivic homotopy theory $SH(k)$ over a field $k$ was introduced by Morel and Voevodsky
in~\cite{MV,VoeICM}. One of its equivalent constructions is given by first stabilizing the triangulated
category of Nisnevich sheaves of $S^1$-spectra $SH^{\nis}_{S^1}(k)$
in the $\gmp$-direction arriving at the triangulated category of bispectra $SH_{\nis}(k)$. Then
$SH(k)$ is defined as the triangulated Bousfield localization of $SH_{\nis}(k)$ with respect to the
full localizing subcategory $\cc S$ compactly generated by cones of the arrows
$\{pr_X:X\times\bb A^1\to X\mid X\in Sm/k\}$. Because of this localization we cannot control
anymore stable homotopy types of bispectra by their stable homotopy sheaves $\pi_{*,*}^{\nis}(A)$.
Instead, we need to compute their ``motivic counterparts" $\pi_{*,*}^{\bb A^1}(A):=\pi_{*,*}^{\nis}(L_{\bb A^1}(A))$,
where $L_{\bb A^1}:SH_{\nis}(k)\to SH_{\nis}(k)$ is the localization functor associated
with $\cc S$.

However, it is enourmosly hard in practice to compute stable motivic homotopy types
(and in particular, the sheaves $\pi_{*,*}^{\bb A^1}(A)$) as well as the
Hom-set $SH(k)(A,B):=SH_{\nis}(k)(L_{\bb A^1}(A),L_{\bb A^1}(B))$ between two bispectra $A,B$.
Therefore the computation
of the localization functor $L_{\bb A^1}$ and the full subcategory of
$L_{\bb A^1}$-local objects is of fundamental importance and requires
a new approach to stable motivic homotopy theory as such.

In~\cite{Voe2} Voevodsky introduced framed correspondences, the main purpose of which
was to suggest such a new approach to stable motivic homotopy theory.
Namely, Voevodsky writes in his notes~\cite{Voe2}: ``I hope that the constructions
described in the notes will lead to a new
model of the stable homotopy theory which will be more friendly for
computations questions. Of course one expects that it will be non-trivial to show
that the new and the old models agree".

In this paper explicit computations of the localization functor $L_{\bb A^1}$
and the associated subcategory of $L_{\bb A^1}$-local objects are given.
As an application, we construct a new explicit and entirely local
model for stable motivic homotopy theory $SH(k)$ predicted by Voevodsky.

In more detail,
using Voevodsky's theory of framed correspondences~\cite{Voe2}, the authors introduce
and develop the machinery of framed motives and big framed motives in~\cite{GP1}.
%that converts the stable motivic homotopy theory of Morel--Voevodsky into
%the local theory of framed bispectra.
In Theorem~\ref{locfunctor} of the present paper,
an explicit computation of the localization functor $L_{\bb A^1}$ is given.
Namely, it is proven that the functor of big framed motives $\cc M^b_{fr}$ in the sense of~\cite{GP1}
determines a localization functor on $SH_{\nis}(k)$ and, moreover, it is isomorphic to $L_{\bb A^1}$.
Also, an explicit description of the full subcategory of $L_{\bb A^1}$-local
objects in $SH_{\nis}(k)$ is given in Theorem~\ref{locobj}.
Namely, it is equivalent to the full subcategory $SH_{\nis}^{fr}(k)$ introduced in
Definition \ref{deffrsp}. As a consequence, the functor
$F: SH_{\nis}^{fr}(k)\to SH(k)$ which is the identity on the objects is an equivalence of categories
(see Theorem \ref{recover}).
Its quasi-inverse is given by the big framed motive functor $\cc M^{b}_{fr}:SH(k)\to SH_{\nis}^{fr}(k)$
(see Theorem \ref{recover}).

Thus the triangulated category of framed bispectra $SH^{fr}_{\nis}(k)$
recovers $SH(k)$. The main feature of $SH^{fr}_{\nis}(k)$ is that its construction is genuinely
local in the sense that it does not use any kind of motivic equivalences. In
other words, we get rid of motivic equivalences completely making
$SH^{fr}_{\nis}(k)$ more amenable to explicit calculations than Morel--Voevodsky's
$SH(k)$. In particular, if $\cc E$, $\cc F$ are two framed bispectra in
$SH_{\nis}^{fr}(k)$, then a morphism $f:\cc E\to\cc F$
is a stable motivic equivalence in the sense of Morel--Voevodsky
if and only if the induced morphisms of Nisnevich sheaves of stable homotopy groups
$f_*:\pi_{*}^{\nis}(\cc E(q))\to\pi_{*}^{\nis}(\cc F(q))$ are isomorphisms in each weight $q$
(see Corollary~\ref{lemstcor}). Therefore stable motivic weak equivalences between framed
bispectra coincide with naive local equivalences between them. Furthermore, let
$SH_{S^1}(k)$ be the stable motivic homotopy category of $S^1$-spectra. Then
the canonical functor
   $$\Omega^\infty_{\gmp}:SH(k)\to SH_{S^1}(k)$$
has the following explicit and elementary computation in our language. Namely, one takes a framed
bispectrum $\cc E$ to its zeroth weight $\cc E(0)$ (see Corollary~\ref{sigmaomega} for details).
A similar explicit computation in terms of framed bispectra is given for the adjoint functor
   $$\Sigma^\infty_{\gmp}:SH_{S^1}(k)\to SH(k)$$
(see Corollary~\ref{sigmaomega}).

%In Theorem~\ref{locfunctor} an explicit computation of the localization functor $L_{\bb A^1}$ is given.
%Namely, it is proven that the functor of big framed motives $\cc M^b_{fr}$ in the sense of~\cite{GP1}
%determines a localization functor in $SH_{\nis}(k)$ and it is isomorphic to $L_{\bb A^1}$.
%Furthermore, $\cc M^b_{fr}$ induces an equivalence of triangulated
%categories $SH_{\nis}(k)/\kr\cc M^b_{fr}\simeq SH_{\nis}^{fr}(k)$ (see Theorem~\ref{locobj}). As a
%consequence,
%$SH_{\nis}^{fr}(k)$ can also be thought of as the triangulated category
%of $\cc M^b_{fr}$-local objects in $SH_{\nis}(k)$.
Since the functor $\cc M^b_{fr}$
determines a localization functor on $SH_{\nis}(k)$, which is isomorphic to $L_{\bb A^1}$, we also get that
$\kr\cc M^b_{fr}$ coincides with
the localizing subcategory $\cc S$ of $SH_{\nis}(k)$
compactly generated by the shifted cones of the arrows
$pr_X:\Sigma^\infty_{S^1}\Sigma^\infty_{\gmp}(X\times\bb
A^1)_+\to\Sigma^\infty_{S^1}\Sigma^\infty_{\gmp}X_+$, $X\in Sm/k$.

Next, the triangulated category of effective framed bispectra $SH^{fr,eff}_{\nis}(k)$
is introduced. It is proved that the big framed motive functor induces
a triangle equivalence
   $$\cc M_{fr}^b:SH^{eff}(k)\to SH^{fr,eff}_{\nis}(k).$$
The reader can find the relevant definitions and statements in Section~\ref{secteff}.

The key ingredient in the theory of framed spectra/bispectra and
motivic infinite loop spaces developed by the authors in~\cite{GP1} is the framed motive $M_{fr}(X)$
of a smooth algebraic variety $X\in Sm/k$. $M_{fr}(-)$ is a functor from smooth
algebraic varieties to sheaves of framed $\bb A^1$-local $S^1$-spectra. In
this paper we define the triangulated category $SH_{S^1}^{fr}[\cc Fr_0(k)]$ of certain functors from
smooth $k$-schemes to sheaves of framed $\bb A^1$-local $S^1$-spectra
and show that $SH_{S^1}^{fr}[\cc Fr_0(k)]$ is naturally equivalent to
$SH(k)$ (see Theorem~\ref{spfuncat}). This result also shows how to construct
the framed motive functor $\cc M^{E}_{fr}\in SH_{S^1}^{fr}[\cc Fr_0(k)]$ out of any
bispectrum $E\in SH(k)$. This construction extends the construction of $M_{fr}(-)$
to all motivic bispectra. Precisely,
$M_{fr}(X)$ is recovered as $\cc M^{E}_{fr}(X)$ with $E$ the motivic
sphere bispectrum. In other words, Morel--Voevodsky stable motivic homotopy theory
is recovered from such framed motives functors.
It is important to note that, by definition, the category
$SH_{S^1}^{fr}[\Fr]$ has nothing to do with any sort of motivic
equivalences. It is purely of local nature.
The reader will find the details in Section~\ref{recspfun}.

It is worth mentioning that computational advantages of framed bispectra and framed spectral
functors introduced and studied in this paper are
crucial for the machinery of motivic $\Gamma$-spaces
developed by authors joint with P.~A.~\O stv\ae r in~\cite{GPO}. Motivic
$\Gamma$-spaces in the sense of~\cite{GPO} extend the celebrated Segal
machine of $\Gamma$-spaces~\cite{S} to the world of motivic homotopy theory.

We finish this paper by introducing the triangulated category of framed
motives $\cc {SH}^{fr}(k)$ and show that the triangulated category of
effective bispectra $SH^{eff}(k)$ is naturally triangle equivalent to
$\cc {SH}^{fr}(k)$ (see Theorem~\ref{spfuneff}).

Throughout the paper $k$ is an infinite perfect field.
We denote by $Sm/k$ the category of smooth
separated schemes of finite type over $k$.
We denote by $(Shv_\bullet(Sm/k),\wedge,pt_+)$ the closed symmetric monoidal
category of pointed Nisnevich sheaves on $Sm/k$. If there is no likelihood of confusion we often write
$[X,Y]$ to denote the internal Hom object $\underline{\Hom}(X,Y)$. The category of
pointed motivic spaces $\bb M_\bullet$ is, by definition, the category
$\Delta^{\op}Shv_\bullet(Sm/k)$ of pointed simplicial sheaves. 
Denote by $\Delta[\bullet]$ the standard cosimplicial simpicial set $n\mapsto\Delta[n]$. If there is no
likelihood of confusion, we sometimes regard it as a cosimplicial smooth scheme, where each
$\Delta[n]$ is regarded as the disjoint union $\bigsqcup_{\Delta[n]}\spec(k)$.

\section{The triangulated category of framed bispectra}\label{frbisp}

Stable motivic homotopy theory $SH(k)$ over a field $k$ was introduced by Morel and Voevodsky
in~\cite{MV,VoeICM}. One of its equivalent constructions is given by first stabilizing the triangulated
category of Nisnevich sheaves of $S^1$-spectra $SH^{\nis}_{S^1}(k)$
in the $\gmp$-direction arriving at the triangulated category of bispectra $SH_{\nis}(k)$.
Let $\cc S$ be the full localizing subcategory in $SH_{\nis}(k)$ compactly generated by cones of the arrows
$\{pr_X:X\times\bb A^1\to X\mid X\in Sm/k\}$.
Let $L_{\bb A^1}:SH_{\nis}(k)\to SH_{\nis}(k)$ be the localization functor associated
with $\cc S$.
Then $\cc S=\kr L_{\bb A^1}$ and
$SH(k)$ is defined as $SH_{\nis}(k)/\cc S$. By the localization theory in
compactly generated triangulated categories~\cite{Nee}, $SH(k)$ is equivalent to the
full subcategory of $L_{\bb A^1}$-local objects.

In this section explicit computations of the localization functor $L_{\bb A^1}$
and the full subcategory of
$L_{\bb A^1}$-local objects are given. As an application, we construct a new explicit and entirely local
model of stable motivic homotopy theory $SH(k)$ predicted by Voevodsky.

%We apply the explicit localization
%functor of big framed motives $\cc M^b_{fr}$ to $SH_{\nis}(k)$.
%Next, we compute the quotient category $SH_{\nis}(k)/\kr\cc
%M^b_{fr}$ or, equivalently saying, the full subcategory of $\cc
%M^b_{fr}$-local objects as the full subcategory $SH_{\nis}^{fr}(k)$
%(see Definition~\ref{deffrsp}) and prove that $SH_{\nis}^{fr}(k)$ is
%equivalent to classical Morel--Voevodsky stable motivic homotopy
%theory $SH(k)$.

%It was shown~\cite{GP1} that Morel--Voevodsky stable motivic
%homotopy category $SH(k)$ is equivalent to the full subcategory
%$SH_{\nis}^{fr}(k)$ consisting of certain framed bispectra. In this
%section we give a genuinely local model for
%$SH_{\nis}^{fr}(k)$ recovering $SH(k)$ in the sense that we shall
%not operate with any type of motivic equivalences in the definition
%of $SH_{\nis}^{fr}(k)$.

Let $Sp_{S^1,\bb G_m}(k)$ be the category of $(S^1,\gmp)$-bispectra
in $\bb M_\bullet$, where $\gmp$ is the mapping cone of the 1-section
$\spec(k)_+\to\bb G_{m+}$ in $\bb M_\bullet$. More precisely, $\gmp$
is the pushout of the zigzag
   $$\spec(k)_+\wedge I\xleftarrow{i_0}\spec(k)_+\to\bb G_{m+},$$
in which $I$ is the pointed simplicial set $\Delta[1]$ with basepoint 1.

The category
$Sp_{S^1,\bb G_m}(k)$ comes equipped with the stable projective
local model structure defined as follows. By~\cite{Bl} $\bb
M_\bullet$ comes equipped with the projective local monoidal model
structure in which weak equivalences are Nisnevich local weak equivalences.
Stabilizing the model structure in the $S^1$-direction, we get the category
$Sp_{S^1}(k)$ of motivic $S^1$-spectra equipped with the stable
projective local monoidal model structure, where weak equivalences
are the maps of spectra inducing isomorphisms on sheaves of stable
homotopy groups. Now stabilising the model structure on
$Sp_{S^1}(k)$ in the $\gmp$-direction, we arrive at the stable
projective local model structure on $Sp_{S^1,\bb G_m}(k)$. Its
triangulated homotopy category is denoted by $SH_{\nis}(k)$.

Given a triangulated category $\cc T$, we define a localization in
$\cc T$ as a triangulated endofunctor $L:\cc T\to\cc T$ together
with a natural transformation $\eta:\id\to L$ such that
$L\eta_X=\eta_{LX}$ for any $X$ in $\cc T$ and $\eta$ induces an
isomorphism $LX\cong LLX$. We refer to $L$ as a {\it localization
functor in $\cc T$}. Such a localization functor determines a full
subcategory $\kr L$ whose objects are those $X$ such that $LX=0$.
An object $X\in\cc T$ is said to be {\it $L$-local\/} if $\eta_X:X\to LX$
is an isomorphism.

The computation of localization functors and their full subcategories
of local objects is enormously hard in practice. In particular, if
$\cc T=SH_{\nis}(k)$ and $\cc S$ is the full subcategory of $SH_{\nis}(k)$ compactly generated
by the shifted cones of the arrows
$pr_X:\Sigma^\infty_{S^1}\Sigma^\infty_{\gmp}(X\times\bb
A^1)_+\to\Sigma^\infty_{S^1}\Sigma^\infty_{\gmp}X_+$, $X\in Sm/k$, then the Bousfield
localization theory in compactly generated triangulated categories~\cite{Nee} says that
there exists a localisation functor
   $$L_{\bb A^1}:SH_{\nis}(k)\to SH_{\nis}(k)$$
such that $\cc S=\kr L_{\bb A^1}$.
%The functor $L_{\bb A^1}$ is defined as the homotopy colimit
%of an infinite tower in $SH_{\nis}(k)$ which is reminiscent to the small object argument.
By definition, the Morel--Voevodsky stable motivic
homotopy category $SH(k)$ is the quotient category $SH_{\nis}(k)/\cc S$.

The computation of the localization functor $L_{\bb A^1}$ and its full subcategory
of local objects is of fundamental importance. The following theorem yields an explicit
computation of $L_{\bb A^1}$.

\begin{thm}\label{locfunctor}
The functor of big framed motives $\cc M^b_{fr}$ in the sense of~\cite{GP1} determines a
localization functor in $SH_{\nis}(k)$. The full subcategory $\kr\cc
M^b_{fr}$ is compactly generated by the shifted cones of the arrows
$pr_X:\Sigma^\infty_{S^1}\Sigma^\infty_{\gmp}(X\times\bb
A^1)_+\to\Sigma^\infty_{S^1}\Sigma^\infty_{\gmp}X_+$, $X\in Sm/k$.
Furthermore, there exists an equivalence of localization functors $\cc M^b_{fr}\simeq L_{\bb A^1}$
and an equivalence of triangulated
categories $SH_{\nis}(k)/\kr\cc M^b_{fr}\simeq SH(k)$.
\end{thm}

We postpone the proof of the theorem. The next goal is to give an explicit computation
of the full subcategory of $L_{\bb A^1}$-local objects in $SH_{\nis}(k)$, which we denote by
$SH^{L_{\bb A^1}}_{\nis}(k)$. To this end,
we introduce and study a full subcategory $SH_{\nis}^{fr}(k)$ in $SH_{\nis}(k)$ which will
give the desired explicit computation of $SH^{L_{\bb A^1}}_{\nis}(k)$ (see Theorem~\ref{locobj}).
To give the definition of $SH_{\nis}^{fr}(k)$ we need some preparations.

We begin with recalling the definition of the category of pointed framed motivic spaces $\mathbb M^{fr}_\bullet$.
%%%and to its natural
%%%enrichment over $\M$.
Let $\mathrm{Fr}_+(k)$ be the category of framed correspondences as in ~\cite[Section~2]{GP1}. Let
$\mathrm{Pre}^{fr}(k)$ be the category of framed presheaves, that is the category of presheaves of sets
on $\mathrm{Fr}_{+}(k)$.
Let $i: \mathrm{Sm}/k\to\mathrm{Sm}/k_{+}\to\mathrm{Fr}_{+}(k)$ be the composite functor.
Recall from~\cite[Section~2]{GP1} that a framed Nisnevich sheaf on $\mathrm{Sm}/k$ is a framed presheaf
such that its restriction to $\mathrm{Sm}/k$
via the functor $i$
is a Nisnevich sheaf. Let $Shv_\bullet^{fr}(k)$ denote the category
of pointed framed Nisnevich sheaves.
The morphisms in this category are just morphisms of pointed framed presheaves.
%%%whose morphisms of sheaves respect framed correspondences.
The {\it category of pointed framed motivic spaces\/} $\mathbb M^{fr}_\bullet$ is the category of simplicial objects in
$Shv_\bullet^{fr}(k)$. There is a canonically induced faithful functor
$\iota:\mathbb M^{fr}_\bullet \rightarrow \mathbb M_\bullet$ obtained from the composite
$i: \mathrm{Sm}/k\to\mathrm{Sm}/k_{+}\to\mathrm{Fr}_{+}(k)$.

Here is some terminology we will use in the paper. 
Given an object $A\in \mathbb M_\bullet$, a {\it framing of $A$} is an object $A'\in \mathbb M^{fr}_\bullet$ such that $A=\iota(A)$.
An object $A\in \mathbb M_\bullet$ is said to be {\it framed\/} provided that a framing $A'$ of $A$ is chosen. 
If $A,B \in \mathbb M_\bullet$ and $A',B'\in \mathbb M^{fr}_\bullet$ are their framings, then
$\Hom_{\mathbb M^{fr}_\bullet}(A',B')\subseteq \Hom_{\mathbb M_\bullet}(A,B)$. 
One says that a morphism $\phi: A\to B$ respects the framings if $\phi$ is in $\Hom_{\mathbb M^{fr}_\bullet}(A',B')$.
%%%If $A'$ is an object of $\mathbb M^{fr}_\bullet$, then the object $A=\iota(A)\in \mathbb M_\bullet$
%%%has a distinguished framing. It is $A'$. 

We claim that for every $A\in \mathbb M_\bullet$ and $B\in \mathbb M^{fr}_\bullet$
the internal Hom-object $\underline{\Hom}_{\bb M_\bullet}(A,\iota(B))$ has a canonical framing.

To show the claim, we need some notation. Following~\cite[Section~6]{Voe2} there is a natural pairing
${Fr}_{0}(k)\times {Fr}_{+}(k)\xrightarrow{\otimes}{Fr}_{+}(k)$
taking $(X,Y)$ to $X\times Y$ and $(f,\alpha)$ to $f\times \alpha$. In what follows
this pairing will be used systematically without referring to it.
Here is the first application of the pairing.
%%%We also use it in the natural enrichment of $\mathbb M^{fr}$ over $\cc M$.

%%%First, we can associate a framed Nisnevich sheaf $\cc F(X\times -)$ to
%%%every framed Nisnevich sheaf $\cc F$ and every $X\in \mathrm{Sm}/k_{+}$. In detail,
%%%given $\alpha\in \mathrm{Fr}_n(U',U)$ put
%%%$\alpha^*: \cc F(X\times U)\to \cc F(X\times U')$
%%%to be $(\id_X\times \alpha)^*$.
%%%If $\cc F$ is a pointed framed Nisnevich sheaf then the framed Nisnevich sheaf
%%%$\cc F(X\times -)$ is pointed also.

%%%Second, every morphism $f: X'\to X$ in $\mathrm{Sm}/k_{+}$ induces a morphism
%%%of framed sheaves $f^*: \cc F(X\times -)\to \cc F(X'\times -)$. Namely, if $U\in \mathrm{Fr}_+(k)$ one sets
%%%$f^*: \cc F(X\times U)\to \cc F(X'\times U)$ to be $(f\times \id_U)^*$.
%%%If $\cc F$ is a pointed framed Nisnevich sheaf, then the morphism of framed sheaves
%%%$f^*: \cc F(X\times -)\to \cc F(X'\times -)$
%%%is a morphism of pointed framed Nisnevich sheaves.

%%%Third, for every $X\in \mathrm{Sm}/k_{+}$
%%%every morphism of framed sheaves $\phi: \cc F\to \cc F'$ induces
%%%a morphism of framed sheaves $\phi_X: \cc F(X\times -)\to \cc F'(X\times -)$.
%%%Namely, if $U\in \mathrm{Fr}_+(k)$ one sets
%%%$\phi_X(U)=\phi(X\times U)$.

For every $B\in \mathbb M^{fr}_\bullet$ the ordinary pointed motivic space
$\iota (B)(X\times\Delta[\bullet]\times -)$ is a contravariant functor in $X$ on the category $Fr_+(k)$.
In detail,
given $\phi\in \mathrm{Fr}_n(X',X)$, an integer $r\geq 0$ and $U\in Sm/k$, put
$\phi^*_{r,U}: \iota (B)(X\times\Delta[r]\times U)\to \iota (B)(X'\times\Delta[r]\times U)$
to be $(\phi \times \id_{\Delta[r]} \times \id_U)^*$.
Denote by $\phi^*_{\bullet,U}$ the family $\phi^*_{r,U}$.
Then $\phi^*_{\bullet,U}: \iota (B)(X\times\Delta[\bullet]\times U)\to \iota (B)(X'\times\Delta[\bullet]\times U)$ is 
a morphism of pointed simplicial sets and the assignment $U\mapsto \phi^*_{\bullet,U}$
determines a morphism between the pointed motivic spaces 
$\iota (B)(X\times\Delta[\bullet]\times -)$ and $\iota (B)(X'\times\Delta[\bullet]\times -)$.
Denote it by $\phi^*_\bullet$.
%%%\begin{lem}
%%%For every $A\in \mathbb M_\bullet$ and $B\in \mathbb M^{fr}_\bullet$
%%%the internal Hom object $\underline{\Hom}(A,B)$ is naturally in $\mathbb M^{fr}_\bullet$.
%%%\end{lem}

Recall that
$\underline{\Hom}(A,\iota (B))(X)$ is the simplicial set $\Hom_{\mathbb M_\bullet}(A,\iota (B)(X\times\Delta[\bullet]\times-))$.
Now the assignment 
   $${Fr}_n(X',X)\ni \phi \mapsto [\Hom_{\mathbb M_\bullet}(A, \phi^*_\bullet): 
       \underline{\Hom}(A,\iota (B))(X)\to \underline{\Hom}(A,\iota (B))(X')]$$
determines a canonical framing on $\underline{\Hom}(A,\iota (B))$ as claimed above.

%%%Finally, similarly to~\eqref{tutu}, $\mathbb M^{fr}$ is naturally enriched over $\cc M$. Namely,
%%%$$
%%%{\mathbb M}(A,B)(X):=\Hom_{\mathbb M^{fr}}(A,B(X\times\Delta[\bullet]\times-)),
%%%\quad A,B\in \mathbb M^{fr}, \ X\in \mathrm{Sm}/k.
%%%$$

Using the terminology above and the claim, the definition of $SH_{\nis}^{fr}(k)$ is as follows.

\begin{dfn}\label{deffrsp}
We define $SH_{\nis}^{fr}(k)$ as a full subcategory in
$SH_{\nis}(k)$ consisting of those bispectra $\cc E$ satisfying the
following conditions:

\begin{enumerate}
\item each motivic space $\cc E_{i,j}$ of $\cc E$ is framed; 
%%%a space with framed correspondences, i.e.
%%%a pointed simplicial sheaf defined on the category of framed correspondences $Fr_*(k)$;

\item the structure maps $\cc E_{i,j}\to\underline{\Hom}(S^1,\cc E_{i+1,j})$,
$\cc E_{i,j}\to\underline{\Hom}(\gmp,\cc E_{i,j+1})$ preserve framings;
%%%framed correspondences;

\item for every $j\geq 0$ the framed presheaves of stable homotopy groups $\pi_*(\cc E_{*,j})$
of the $S^1$-spectrum $\cc E_{*,j}$ are stable, radditive and $\bb
A^1$-invariant;

\item (Cancellation Theorem) for every $j\geq 0$ the structure map
$\cc E_{*,j}\to\underline{\Hom}(\gmp,\cc E_{*,j+1})$ of $S^1$-spectra is a stable local equivalence.
\end{enumerate}
\end{dfn}

We should stress that the definition of $SH_{\nis}^{fr}(k)$ is
local in the sense that its morphisms are computed in
$SH_{\nis}(k)$ and has nothing to do with $SH(k)$.

The following theorem yields an explicit description of the full subcategory of
$L_{\bb A^1}$-local objects $SH^{L_{\bb A^1}}_{\nis}(k)$.

\begin{thm}\label{locobj}
Every object of $SH_{\nis}^{fr}(k)$ is $L_{\bb A^1}$-local and the inclusion functor
$SH_{\nis}^{fr}(k)\to SH^{L_{\bb A^1}}_{\nis}(k)$ is an equivalence of triangulated categories.
Furthermore, the functor $\cc M^{b}_{fr}:SH_{\nis}(k)\to SH_{\nis}(k)$ induces an equivalence of triangulated
categories $SH_{\nis}(k)/\kr\cc M^b_{fr}\simeq SH_{\nis}^{fr}(k)$.
\end{thm}

We also have the following

\begin{thm}\label{recover}
The natural functor $F:SH_{\nis}^{fr}(k)\to SH(k)$, which is the
identity on objects, is an equivalence of categories. Its
quasi-inverse is given by the big framed motive functor $\cc
M^{b}_{fr}:SH(k)\to SH_{\nis}^{fr}(k)$ in the sense of~\cite{GP1}.
\end{thm}

We postpone the proofs of Theorems~\ref{locobj} and~\ref{recover} but first we need some preparations.

\begin{ntt}\label{weight}
Given a bispectrum $\cc E\in Sp_{S^1,\bb G_m}(k)$, we shall also
write $\cc E=(\cc E(0),\cc E(1),\ldots)$, where $\cc E(j)$ is the
motivic $S^1$-spectrum with spaces $\cc E(j)_i:=\cc E_{i,j}$, $i\geq
0$. We also call $\cc E(j)$ the {\it $j$-th weight of $\cc E$}.
\end{ntt}

For the convenience of the reader we recall
from~\cite[Section~12]{GP1} the definition of the big framed motive.
For any bispectrum $E$, let $E^c$ be its cofibrant replacement in
the projective model structure. Then $E^c$ consists of motivic
spaces $E^c_{i,j}$ which are sequential colimits of simplicial
smooth $k$-schemes. We then take $C_*Fr(E^c_{i,j})$ at every entry
and get a bispectrum $C_*Fr(E^c)$, where $C_*Fr$ is the canonical
functor from motivic spaces to $\bb A^1$-invariant framed motivic
spaces (see~\cite{GP1} for details). Finally, the big framed motive
$\cc M^{b}_{fr}(E)$ is, by definition, the bispectrum
$\Theta^\infty_{\gmp}\Theta_{S^1}^\infty C_*Fr(E^c)$ obtained from
$C_*Fr(E^c)$ by the stabilization in the $S^1$- and
$\gmp$-directions. $\cc M^{b}_{fr}(E)$ is functorial in $E$ and
converts stable motivic equivalences of bispectra into level local
equivalences (see~\cite[12.4]{GP1}). Moreover, the canonical arrow
   \begin{equation*}\label{alphaE}
    \alpha_E:E^c\to\cc M^{b}_{fr}(E)
   \end{equation*}
is a stable motivic equivalence by~\cite[12.1-12.2]{GP1}. Also, the zigzag
   \begin{equation}\label{zigzagE}
    \eta_E:E\xleftarrow{\tau_E}E^c\xrightarrow{\alpha_E}\cc M^{b}_{fr}(E)
   \end{equation}
gives a natural transformation of endofunctors in $SH_{\nis}(k)$
   \begin{equation}\label{nattran}
    \eta:\id\to\cc M^{b}_{fr},
   \end{equation}
because $\tau_E$ is a level weak equivalence by construction. Note
that $\eta$ can also be regarded as a natural transformation of
endofunctors in $SH(k)$, in which case it is an isomorphism.

If there is no likelihood of confusion, we will also regard $\gmp$ as a simplicial smooth scheme
from $\Delta^{\op}Fr_0(k)$, where $Fr_0(k)$ is the category of framed correspondences of level 0
in the sense of Voevodsky~\cite{Voe2}.
Recall that the objects of $Fr_0(k)$ are those of $Sm/k$ and the morphism sets
$\Hom_{Shv_\bullet(Sm/k)}(U_+,V_+)$.

Here are some examples of objects of $SH_{\nis}^{fr}(k)$.

\begin{ex}
(1) A typical example of an object of $SH_{\nis}^{fr}(k)$ is the
bispectrum
   $$M_{fr}^{\bb G}(X)=(M_{fr}(X),M_{fr}(X\times\gmp),M_{fr}(X\times\bb G^{\wedge 2}_m),\ldots),\quad X\in Sm/k,$$
consisting of twisted framed motives of $X$ (see~\cite{GP1} for
details). Recall from~\cite{GP1} that $M_{fr}(X):=C_*Fr(\Sigma^\infty_{S^1}X_+)$.
By~\cite[11.1]{GP1} the canonical morphism of bispectra
$\Sigma^\infty_{S^1}\Sigma^\infty_{\gmp}X_+\to M_{fr}^{\bb G}(X)$ is
an isomorphism in $SH(k)$.

(2) More generally, let $A$ be an $S^1$-spectrum such that every entry $A_i$
of $A$ is a sequential colimit of $k$-smooth simplicial schemes. Then
$C_*Fr(\Sigma^\infty_{\gmp}A)$ is in $SH^{fr}_{\nis}(k)$ (we use here Lemma~\ref{yyy}).

(3) Another example of an object of $SH_{\nis}^{fr}(k)$ is the
bispectrum
   $$M^{\bb G}(X):=(M(X),M(X\times\gmp),M(X\times\bb G^{\wedge 2}_m),\ldots),\quad X\in Sm/k,$$
where each weight $M(X\times\bb G^{\wedge j}_m)$ is the
Eilenberg--Mac~Lane $S^1$-spectrum associated with the Voevodsky's
motivic complex $C_*\bb Z_{tr}(X\times\bb G^{\wedge j}_m)$. Note
that $M^{\bb G}(pt)$ represents motivic cohomology in $SH(k)$.

(4) More generally, let $\cc A$ be a strict $V$-category of
correspondences in the sense of~\cite{GG} admitting a functor of
categories $F:Fr_*(k)\to\cc A$ such that $F$ is identity on objects
and $F(\sigma_X)=\id_X$ for all $X\in Sm/k$. Then the bispectrum
   $$M_{\cc A}^{\bb G}(X)=(M_{\cc A}(X),M_{\cc A}(X\times\gmp),M_{\cc A}(X\times\bb G^{\wedge 2}_m),\ldots),$$
in which each weight $M_{\cc A}(X\times\bb G^{\wedge j}_m)$ is the
Eilenberg--Mac~Lane $S^1$-spectrum associated with the complex
$C_*\cc A(-,X\times\bb G^{\wedge j}_m)_{\nis}$, is an object of
$SH_{\nis}^{fr}(k)$.
\end{ex}

\begin{lem}\label{lemprem}
Let $\cc X$ be a $S^1$-spectrum with presheaves of stable homotopy
groups $\pi_*(\cc X)$ being framed stable radditive and $\bb
A^1$-invariant. Let $\cc X_f$ be a spectrum obtained from $\cc X$ by
taking a local stable fibrant replacement in the stable local
projective model structure of $Sp_{S^1}(k)$. Then $\cc X_f$ is
motivically fibrant.
\end{lem}

\begin{proof}
We have $\cc X_f=\hocolim_{n\to-\infty}(\cc X_{\geq n})_f$, where
$(\cc X_{\geq n})_f$ is the $n$th truncation of $\cc X$ in
$Sp_{S^1}(k)$. $\cc X$ has homotopy invariant, stable radditive
presheaves with framed correspondences of stable homotopy groups
$\pi_*(\cc X)$. If
$(\cc X_{\geq n})_f$ is a stable local replacement of $\cc X_{\geq
n}$, the proof of~\cite[7.4]{GP1} shows that $(\cc X_{\geq n})_f$ is
motivically fibrant, and hence so is $\cc X_f$.
\end{proof}

\begin{lem}\label{lemst}
Let $\cc E=(\cc E(0),\cc E(1),\ldots)$ be a bispectrum in
$SH_{\nis}^{fr}(k)$ and let $\cc E_f=(\cc E(0)_f,\cc E(1)_f,\ldots)$ be a
bispectrum obtained from $\cc E$ by taking a local stable fibrant
replacement of each weight $\cc E(j)$ in the stable local projective
model structure of $Sp_{S^1}(k)$. Then $\cc E_f$ is motivically fibrant.
\end{lem}

\begin{proof}
By Lemma~\ref{lemprem} each $S^1$-spectrum $\cc E(j)_f$ is
motivically fibrant. Consider a commutative diagram of $S^1$-spectra
   $$\xymatrix{\cc E(j)\ar[r]\ar[d]&\underline{\Hom}(\gmp,\cc E(j+1))\ar[d]\\
               \cc E(j)_f\ar[r]&\underline{\Hom}(\gmp,\cc E(j+1)_f)}$$
The upper arrow and the left vertical arrow are stable local
equivalences by assumption. The right vertical arrow is a stable
local equivalence by the sublemma of~\cite[Section~12]{GP1}. We see
that the lower arrow is a stable local equivalence between
motivically fibrant $S^1$-spectra. We conclude that it is a
sectionwise weak equivalence as well, and hence $\cc E$ is a motivically
fibrant bispectrum.
\end{proof}

If $E$ is a bispectrum and $p,q$ are integers, recall that $\pi_{p,q}^{\bb A^1}(E)$ is the
sheaf associated to the presheaf
   $$U\longmapsto SH(k)(\Sigma^\infty_{\gmp}\Sigma^\infty_{S^1}U_+\wedge S^{p-q}\wedge\bb G_m^{\wedge q},E).$$
Given a presheaf of Abelian groups $F$, denote by $F_{-1}$ the presheaf mapping $U\in Sm/k$
to the kernel of the evaluation at $1: F(U\times\bb G_m)\to F(U)$. The presheaf $F_{-q}$, $q>1$,
is defined recursively.

The next statement says that the sheaves $\pi_{*,*}^{\bb A^1}(E)$, where $E$ is
a framed bispectrum, are computed
in terms of ordinary Nisnevich sheaves of stable homotopy groups for weighted $S^1$-spectra of $E$.

\begin{cor}\label{lemstcor}
Let $\cc E=(\cc E(0),\cc E(1),\ldots)$ be a bispectrum in
$SH_{\nis}^{fr}(k)$ and $p,q\in\bb Z$. Then $\pi_{p,q}^{\bb A^1}(\cc E)=\pi_{p-q}^{\nis}(\cc E(|q|))$ if $q\leq 0$
and $\pi_{p,q}^{\bb A^1}(\cc E)=(\pi_{p-q}^{\nis}(\cc E(0)))_{-q}$ if $q>0$, where $|q|$ is the modulus of $q$.
In particular, if $\cc F=(\cc F(0),\cc F(1),\ldots)$ is another bispectrum in
$SH_{\nis}^{fr}(k)$, then an ordinary morphism of motivic bispectra $f:\cc E\to\cc F$
is a stable motivic equivalence in the sense of Morel--Voevodsky
if and only if the induced morphisms of Nisnevich sheaves
$f_*:\pi_{*}^{\nis}(\cc E(q))\to\pi_{*}^{\nis}(\cc F(q))$ are isomorphisms in each weight $q$.
\end{cor}

\begin{proof}
Lemma~\ref{lemst} implies $\pi_{p,q}^{\bb A^1}(\cc E)=\pi_{p-q}^{\nis}(\cc E(|q|))$ if $q\leq 0$
and $\pi_{p,q}^{\bb A^1}(\cc E)=\pi_{p-q}^{\nis}(\Omega_{\bb G_m^{\wedge q}}(\cc E(0)_f))$ if $q>0$.
The proof of the sublemma in~\cite[Section~12]{GP1} shows that
   $$\pi_{p-q}^{\nis}(\Omega_{\bb G_m^{\wedge q}}(\cc E(0)_f))
     =(\pi_{p-q}^{\nis}(\cc E(0)_f))_{-q}=(\pi_{p-q}^{\nis}(\cc E(0)))_{-q},$$
as required.
\end{proof}

\begin{proof}[Proof of Theorem~\ref{recover}]
Given $\cc E,\cc E'\in SH_{\nis}^{fr}(k)$, one has
   $$\Hom_{SH(k)}(\cc E,\cc E')=\Hom_{Sp_{S^1,\bb G_m}(k)}(\cc E^c,\cc E'_{m,f})/\sim,$$
where $\sim$ refers to the naive homotopy, $\cc E^c$ is a cofibrant
replacement of $\cc E$ and $\cc E'_{m,f}$ is a motivically fibrant
replacement of $\cc E'$. By Lemma~\ref{lemst} $\cc E'_{m,f}$ can be
computed as $\cc E'_{f}$, i.e. as a level stable local fibrant
replacement of $\cc E'$. Since $\cc E^c_{i,j}\to\cc E_{i,j}$ is a
weak equivalence in $\bb M_*$ for all $i,j\geq 0$ and
   $$\Hom_{SH^{fr}_{\nis}(k)}(\cc E,\cc E')=\Hom_{Sp_{S^1,\bb G_m}(k)}(\cc E^c,\cc E'_{f})/\sim,$$
it follows that $F:SH_{\nis}^{fr}(k)\to SH(k)$ is fully faithful.

Given any $E\in SH(k)$ let $\cc M^{b}_{fr}(E)$ be its big framed
motive in the sense of~\cite{GP1}. Then $\cc M^{b}_{fr}(E)\in
SH^{fr}_{\nis}(k)$ and the zigzag $\eta_E$~\eqref{zigzagE} yields an
isomorphism $E\cong \cc M^{b}_{fr}(E)$ in $SH(k)$. We conclude that
$F$ is an equivalence of categories, as was to be shown.
\end{proof}

Denote by $SH_{S^1}(k)$ the stable motivic homotopy category of $S^1$-spectra.
There is a canonical pair of adjoint functors
   $$\Sigma^\infty_{\gmp}:SH_{S^1}(k)\rightleftarrows SH(k):\Omega^\infty_{\gmp}.$$
The following result is a consequence of Theorem~\ref{recover} and~\cite[12.1]{GP1}.

\begin{cor}\label{sigmaomega}
Let $F:SH_{\nis}^{fr}(k)\to SH(k)$ be the equivalence of Theorem~\ref{recover}. Then
the composite functor $\Omega^\infty_{\gmp}\circ F:SH_{\nis}^{fr}(k)\to SH_{S^1}(k)$
is equivalent to the functor taking a bispectrum $\cc E=(\cc E(0),\cc E(1),\ldots)\in SH_{\nis}^{fr}(k)$
to its zeroth weight $\cc E(0)$. In turn, the composite functor
$\cc M_{fr}^b\circ\Sigma^\infty_{\gmp}:SH_{S^1}(k)\to SH_{\nis}^{fr}(k)$
is equivalent to the functor taking a motivic $S^1$-spectrum $E\in SH_{S^1}(k)$ to the
framed bispectrum $C_*Fr(\Sigma^\infty_{\gmp}E^c)$.
\end{cor}

\begin{proof}[Proof of Theorem~\ref{locfunctor}]
We define a natural transformation $\eta:\id\to\cc M^b_{fr}$ as
in~\eqref{nattran}. Consider a commutative diagram of bispectra
   $$\xymatrix{E&E^c\ar[l]_{\tau_E}\ar[r]^{\alpha_E}&\cc M^b_{fr}(E)\\
               E^c\ar[u]^{\tau_E}\ar[d]_{\alpha_E}
               &(E^c)^c\ar[u]^{\tau_{E^c}}\ar[r]^{(\alpha_{E})^c}\ar[l]_{(\tau_{E})^c}\ar[d]_{\alpha_{E^c}}
               &\cc M^b_{fr}(E)^c\ar[d]^{\alpha_{\cc M^b_{fr}(E)}}\ar[u]_{\tau_{\cc M^b_{fr}(E)}}\\
               \cc M^b_{fr}(E)&\cc M^b_{fr}(E^c)\ar[l]^{\cc M^b_{fr}(\tau_E)}\ar[r]_(.4){\cc M^b_{fr}(\alpha_E)}
               &\cc M^b_{fr}(\cc M^b_{fr}(E))}$$
The left vertical and upper horizontal zigzags equal $\eta_E$. The
lower horizontal zigzag equals $\cc M^b_{fr}(\eta_E)$. In turn, the
right vertical zigzag equals $\eta_{\cc M^b_{fr}(E)}$. It follows
that $\eta_{\cc M^b_{fr}(E)}\circ\eta_E$ equals $\cc
M^b_{fr}(\eta_E)\circ\eta_E$ in $\Mor(SH(k))$. As we have noticed above,
$\eta_E$ is an isomorphism in $SH(k)$, and hence $\eta_{\cc
M^b_{fr}(E)}=\cc M^b_{fr}(\eta_E)$ in $SH(k)$. Since the
functor $F:SH_{\nis}^{fr}(k)\to SH(k)$ of Theorem~\ref{recover} is
an equivalence, we see that $\eta_{\cc M^b_{fr}(E)}=\cc
M^b_{fr}(\eta_E)$ in $SH_{\nis}^{fr}(k)$, because both zigzags are
images of $F$. Since $\cc M^b_{fr}$
converts stable motivic equivalences to level local equivalences
by~\cite[12.4]{GP1}, it follows that $\eta_{\cc M^b_{fr}(E)}$, $\cc
M^b_{fr}(\eta_E)$ are zigzags of level local equivalences.
We have also verified that $\eta$ induces an
isomorphism $\cc M^b_{fr}(E)\cong\cc M^b_{fr}(\cc M^b_{fr}(E))$. So
$\cc M^b_{fr}$ determines a localization functor in $SH_{\nis}(k)$.

Next, let $\cc S$ denote the full subcategory of $SH_{\nis}(k)$
compactly generated by the shifted cones of the arrows
$pr_X:\Sigma^\infty_{S^1}\Sigma^\infty_{\gmp}(X\times\bb
A^1)_+\to\Sigma^\infty_{S^1}\Sigma^\infty_{\gmp}X_+$, $X\in Sm/k$.
By definition, $SH(k)$ is the quotient category $SH_{\nis}(k)/\cc
S$. Therefore, a bispectrum $E$ is isomorphic to zero in $SH(k)$ if
and only if it is in $\cc S$. By~\cite[12.4]{GP1} $\eta:\id\to\cc
M^b_{fr}$ is an isomorphism of endofunctors in $SH(k)$. We see
that $\kr\cc M^b_{fr}=\cc S=\kr L_{\bb A^1}$. It follows that $\cc M^b_{fr}$
induces an equivalence of triangulated categories
$SH_{\nis}(k)/\kr\cc M^b_{fr}\simeq SH(k)$ and there is an equivalence
of localization functors $\cc M^b_{fr}\simeq L_{\bb A^1}$.
\end{proof}

\begin{proof}[Proof of Theorem~\ref{locobj}]
Let $\cc E$ be in $SH_{\nis}^{fr}(k)$. Then $\cc E$ is $L_{\bb A^1}$-local if
and only if $\Hom_{SH_{\nis}}(k)(S,\cc E)=0$ for all $S\in\cc S$. We use here basic facts of
localization theory in compactly generated triangulated categories~\cite{Nee}. By this theory
the latter equality is true if and only if $\Hom_{SH_{\nis}}(k)(T,\cc E)=0$ for all compact generators
$T\in\cc S$. Since $\cc E_f$ is a motivically fibrant bispectrum by
Lemma~\ref{lemst}, it follows that $\Hom_{SH_{\nis}}(k)(T,\cc E)=\Hom_{SH_{\nis}}(k)(T,\cc E_f)=0$.
So $\cc E$ is $L_{\bb A^1}$-local.

Next, Theorem~\ref{locfunctor} implies that any $L_{\bb A^1}$-local object
$X$ is $\cc M^b_{fr}$-local, hence $X\cong\cc M^b_{fr}(X)\in SH_{\nis}^{fr}(k)$. It follows that
the inclusion functor $SH_{\nis}^{fr}(k)\to SH_{\nis}^{L_{\bb A^1}}(k)$ is an equivalence of categories. The fact that
$\cc M^b_{fr}$ induces an equivalence of triangulated categories $SH_{\nis}(k)/\kr\cc M^b_{fr}\simeq SH_{\nis}^{fr}(k)$
follows from Theorems~\ref{locfunctor} and~\ref{recover}.
\end{proof}

\begin{cor}\label{loccor}
The functor of big framed motives $\cc M^b_{fr}:SH_{\nis}(k)\to
SH_{\nis}^{fr}(k)$ is left adjoint to the inclusion functor
$\iota:SH_{\nis}^{fr}(k)\hookrightarrow SH_{\nis}(k)$.
\end{cor}

We can summarize the results of the section as follows. We start
with the local stable homotopy category of sheaves of $S^1$-spectra
$SH_{S^1}^{\nis}(k)$, which is also closed symmetric monoidal
triangulated. Then stabilizing $SH_{S^1}^{\nis}(k)$ with respect to
the endofunctor $\gmp\wedge-$, we arrive at the triangulated
category $SH_{\nis}(k)$. We apply the explicit localization
functor of big framed motives $\cc M^b_{fr}$ to $SH_{\nis}(k)$.
Next, we compute the quotient category $SH_{\nis}(k)/\kr\cc
M^b_{fr}$ or, equivalently saying, the full subcategory of $\cc
M^b_{fr}$-local objects as the full subcategory $SH_{\nis}^{fr}(k)$
(see Definition~\ref{deffrsp}) and prove that $SH_{\nis}^{fr}(k)$ is
equivalent to classical Morel--Voevodsky stable motivic homotopy
theory $SH(k)$.

\section{Effective framed bispectra}\label{secteff}

In~\cite{Lev} Levine computes slices of motivic $S^1$-spectra and motivic bispectra.
Basing on these computations, Bachmann--Fasel~\cite{BF} give a criterion for effective motivic bispectra.
In this section we use Bachmann--Fasel's~\cite{BF} technique to describe effective framed bispectra
(see Definition~\ref{dfneff} and Theorem~\ref{thmeff}).
Throughout this section by a semi-local scheme we shall mean a localisation of a smooth irreducible
scheme at finitely many closed points.

\begin{lem}\label{psv}
Let $\cc F$ be a framed radditive $\bb A^1$-invariant quasi-stable presheaf of Abelian groups. Then
$\cc F(U)\cong\cc F_{\nis}(U)$ for any semi-local scheme $U$.
\end{lem}

\begin{proof}
We claim that the restriction map $\cc F(U)\to\cc F(k(U))$ is injective. This was shown for local
schemes in~\cite{GP4}. The same proof works over semi-local schemes if we use~\cite[2.2, 2.3, 4.3]{PSV}.
Now consider a commutative diagram with exact rows
   $$\xymatrix{0\ar[r]&(\kr\alpha)(U)\ar[r]\ar[d]&\cc F(U)\ar[d]\ar[r]^{\alpha(U)}&\cc F_{\nis}(U)\ar[d]\ar[r]&(\coker\alpha)(U)\ar[d]\ar[r]&0\\
                      0\ar[r]&(\kr\alpha)(k(U))\ar[r]&\cc F(k(U))\ar[r]^{\alpha(k(U))}&\cc F_{\nis}(k(U))\ar[r]&(\coker\alpha)(k(U))\ar[r]&0}$$
with $\alpha:\cc F\to\cc F_{\nis}$ the canonical sheafification map. Since $\cc F_{\nis}$ is a
framed radditive $\bb A^1$-invariant quasi-stable presheaf of Abelian groups by~\cite{GP4}, then so are
the presheaves $\kr\alpha,\coker\alpha$. It follows that the vertical maps of the diagram are monomorphisms.
But $\alpha(k(U))$ is an isomorphism, and hence $(\kr\alpha)(k(U))=(\coker\alpha)(k(U))=0$. We see that
$(\kr\alpha)(U)=(\coker\alpha)(U)=0$, and so $\alpha(U)$ is an isomorphism.
\end{proof}

\begin{lem}\label{semiloc}
Let $\cc X\in Sp_{S^1}(k)$ be a motivic $S^1$-spectrum with presheaves of stable homotopy groups being
framed, radditive, quasi-stable and $\bb A^1$-invariant. Suppose $\alpha:\cc X\to\cc X_f$ is a local stable fibrant replacement of $\cc X$.
Then the induced map of $S^1$-spectra $\alpha(U):\cc X(U)\cong\cc X_f(U)$ is a stable equivalence of
ordinary spectra for any semi-local scheme $U$.
\end{lem}

\begin{proof}
For every $i\in\bb Z$, consider a commutative diagram of Abelian groups
   $$\xymatrix{(\pi_i(\cc X))(U)\ar[r]\ar[d]&(\pi_i(\cc X_f))(U)\ar[d]\\
                       (\pi_i^{\nis}(\cc X))(U)\ar[r]^\cong&(\pi_i^{\nis}(\cc X_f))(U).}$$
By Lemma~\ref{psv} the left vertical map is an isomorphism. The lower arrow is an
isomorphism for obvious reasons. Therefore our assertion will be proved if we show
that the right vertical arrow is an isomorphism.

Suppose first $\cc X$ is $(-1)$-connected. Consider the Brown--Gersten spectral sequence
   $$H_{\nis}^p(U,\pi_q^{\nis}(\cc X))\Longrightarrow(\pi_{q-p}(\cc X_f))(U).$$
Since $U$ is semi-local, $H_{\nis}^p(U,\pi_q^{\nis}(\cc X))=0$ for all $p>0$. To see this, we use the fact that
$V\mapsto H_{\nis}^p(V,\pi_q^{\nis}(\cc X))$ is a framed quasi-stable radditive $\bb A^1$-invariant presheaf
by~\cite[Section~16]{GP4} and the restriction map $H_{\nis}^p(U,\pi_q^{\nis}(\cc X))\hookrightarrow H_{\nis}^p(k(U),\pi_q^{\nis}(\cc X))$
is a monomorphism by the proof of Lemma~\ref{psv}. It follows that
$(\pi_{q}(\cc X_f))(U)\cong(\pi_{q}^{\nis}(\cc X_f))(U)$.

If $\cc X$ is not connected, then it is sectionwise weakly equivalent to $\hocolim_{n\to-\infty}\cc X_{\geq n}$, where
$\cc X_{\geq n}$ is the $n$th truncation of $\cc X$ in the category of presheaves of $S^1$-spectra (see~\cite[\S1.6]{Lev}). Moreover,
$\cc X_f$ is sectionwise weakly equivalent to $\hocolim_{n\to-\infty}(\cc X_{\geq n})_f$, where $(\cc X_{\geq n})_f$
is a fibrant replacement of $\cc X_{\geq n}$ in the stable local projective model structure of presheaves of $S^1$-spectra.
Then we have $\pi_i(\cc X_f)=\colim\pi_i((\cc X_{\geq n})_f)$ and $\pi_i^{\nis}(\cc X_f)=\colim\pi_i^{\nis}((\cc X_{\geq n})_f)$.
As above, $\pi_i((\cc X_f)_{\geq n})(U)\cong\pi_i^{\nis}((\cc X_f)_{\geq n})(U)$,
and hence
   $$\pi_i(\cc X_f)(U)=\colim\pi_i((\cc X_f)_{\geq n})(U)\cong\colim\pi_i^{\nis}((\cc X_f)_{\geq n})(U)
       =\pi_i^{\nis}(\cc X_f)(U),$$
as required.
\end{proof}

Given a field $K$ and $\ell\geq 0$, let $\cc O(\ell)_{K,v}$
denote the semi-local ring of the set $v$ of vertices of $\Delta^\ell_{K}=\spec (K[t_0,\ldots,t_\ell]/(t_0+\cdots+t_\ell-1))$ and set
   $$\wh{\Delta}^\ell_{K}:=\spec\cc O(\ell)_{K,v}.$$
Then $\ell\mapsto\wh{\Delta}^\ell_{K}$ is a cosimplicial semi-local subscheme of ${\Delta}^\bullet_{K}$.

Let $E$ be an $\bb A^1$-invariant and Nisnevich excisive $S^1$-spectrum and let $s_0(E)$ be its
zeroth slice. The following lemma computes its sections $s_0(E)(Y)$, $Y\in Sm/k$, where $s_0(E)(Y)$ is,
by definition, the value of a motivically fibrant replacement of $s_0(E)$ at $Y$.

\begin{lem}(\cite[2.2.6]{KL})\label{kls0}
For $Y\in Sm/k$, $s_0(E)(Y)$ is weakly equivalent to the total spectrum $E(\wh{\Delta}^\bullet_{k(Y)})$
of the simplicial spectrum $\ell\mapsto E(\wh{\Delta}^\ell_{k(Y)})$.
\end{lem}

\begin{cor}\label{kls0cor}
$s_0(E)=0$ in $SH_{S^1}(k)$ if and only if the total spectrum $E(\wh{\Delta}^\bullet_{K})=0$ in $SH$
for all finitely generated fields $K/k$.
\end{cor}

\begin{proof}
Suppose $s_0(E)=0$ in $SH_{S^1}(k)$. Then $s_0(E)(Y)=0$ in $SH$ for all $Y\in Sm/k$.
Let $K/k$ be a finitely generated field. Since the base field
$k$ is perfect, then $K=k(Y)$ for some $Y\in Sm/k$. By Lemma~\ref{kls0} the
total spectrum $E(\wh{\Delta}^\bullet_{K})=0$ in $SH$.
The converse is obvious if we use Lemma~\ref{kls0}.
\end{proof}

\begin{dfn}\label{dfneff}
Let $\cc E=(\cc E(0),\cc E(1),\ldots)\in SH_{\nis}^{fr}(k)$ be a framed bispectrum in the sense of Definition~\ref{deffrsp}.
We say that $\cc E$ is {\it effective\/} if the total spectrum $\cc E(j)(\wh{\Delta}^\bullet_{K})$ is stably trivial, i.e. equals zero
in $SH$, for all positive weights $j>0$ and all finitely generated fields $K/k$. The full subcategory of $SH_{\nis}^{fr}(k)$ of
effective framed bispectra will be denoted by $SH_{\nis}^{fr,\,eff}(k)$.
\end{dfn}

The following result is a framed analog of Bachmann--Fasel's theorem~\cite[4.4]{BF}.

\begin{thm}\label{thmeff}
A framed bispectrum $\cc E\in SH_{\nis}^{fr}(k)$ is effective in the sense of Definition~\ref{dfneff}
if and only if it is effective
as an ordinary motivic bispectrum in $SH(k)$, i.e. $\cc E\in SH^{eff}(k)$.
\end{thm}

\begin{proof}
We begin with the following observation.
%\begin{lem}
A bispectrum $E\in SH^{eff}(k)$ satisfies $\Omega^\infty_{\gmp}(E)=0$ in $SH_{S^1}(k)$
if and only if $E=0$ in $SH(k)$.
%\end{lem}
Indeed, this follows from the
isomorphism
   $$SH_{S^1}(k)(\Sigma^\infty_{S^1}U_+[i],\Omega^\infty_{\gmp}(E))\cong
       SH(k)(\Sigma^\infty_{\gmp}\Sigma^\infty_{S^1}U_+[i],E),\quad i\in\bb Z,\ U\in Sm/k,$$
and the fact that the objects
$\Sigma^\infty_{\gmp}\Sigma^\infty_{S^1}U_+[i]$ (respectively
$\Sigma^\infty_{S^1}U_+[i]$) are compact generators of the triangulated category $SH^{eff}(k)$
(respectively $SH_{S^1}(k)$).

Suppose $\cc E\in SH_{\nis}^{fr}(k)$ is effective. We claim that all
negative slices $s_{n<0}(\cc E)$ are zero in $SH(k)$. As usual, let
$\cc E_f=(\cc E(0)_f,\cc E(1)_f,\ldots)$ be the motivic bispectrum
obtained from $\cc E$ by taking stable projective local fibrant
replacements of $S^1$-spectra levelwise. By Lemma~\ref{lemst} $\cc
E_f$ is motivically fibrant. We have that $\cc E\wedge\gmpn\cong(\cc
E(n)_f,\cc E(n+1)_f,\ldots)$ for all $n>0$ and $s_0(\cc
E\wedge\gmpn) \cong s_{-n}(\cc E)\wedge\gmpn$ in $SH(k)$. Since
$-\wedge\gmp$ is an autoequivalence of $SH(k)$, it follows that
$s_{-n}(\cc E)=0$ if and only if $s_0((\cc E(n)_f,\cc
E(n+1)_f,\ldots))=0$ in $SH(k)$.

By Lemma~\ref{semiloc} $\alpha:\cc E(n)(\wh{\Delta}^\ell_K)\to\cc
E(n)_f(\wh{\Delta}^\ell_K)$ is a stable equivalence of ordinary
$S^1$-spectra for every $\ell\geq 0$ and every finitely generated
field $K/k$. It follows that the map of total spectra $\alpha:\cc
E(n)(\wh{\Delta}^\bullet_K)\to\cc E(n)_f(\wh{\Delta}^\bullet_K)$ is
a stable equivalence. But $\cc E(n)(\wh{\Delta}^\bullet_K)$ is
stably trivial by assumption, and hence so is $\cc
E(n)_f(\wh{\Delta}^\bullet_K)$. By Corollary~\ref{kls0cor}
$s_0(\cc E(n)_f)=0$ in $SH_{S^1}(k)$.
%Since $\cc E_f$ is motivically fibrant one has
%$\cc E(n)_f=\Omega^\infty_{\gmp}((\cc E(n)_f,\cc E(n+1)_f,\ldots))$
%We see that $s_{-n}(\cc E)=0$ in
%$SH(k)$ for all $n>0$ as claimed.

It follows from~\cite[7.1.1, 9.0.3]{Lev} (see~\cite[4.1]{BF} as
well) that for all $E\in SH(k)$, $\Omega^\infty_{\gmp}(s_0(E))=
s_0(\Omega^\infty_{\gmp}(E))$ in $SH_{S^1}(k)$.
Using this and the fact that $\cc E_f$ is motivically fibrant we have
$$s_0(\cc E(n)_f)=s_0(\Omega^\infty_{\gmp}(\cc E(n)_f,\cc E(n+1)_f,\ldots))=\Omega^\infty_{\gmp}(s_0((\cc E(n)_f,\cc E(n+1)_f,\ldots))).$$
Applying the above observation to the bispectrum
$E=s_0((\cc E(n)_f,\cc E(n+1)_f,\ldots)$ and using the equality
$s_0(\cc E(n)_f)=0$ in $SH_{S^1}(k)$
we conclude that
$0=s_0((\cc E(n)_f,\cc E(n+1)_f,\ldots)$.
%Observe that $E\in
%SH^{eff}(k)$ satisfies $\Omega^\infty_{\gmp}(E)=0$ in $SH_{S^1}(k)$
%if and only if $E=0$ in $SH(k)$. Indeed, this follows from the
%isomorphism
%   $$SH_{S^1}(k)(\Sigma^\infty_{S^1}U_+[i],\Omega^\infty_{\gmp}(E))\cong
%       SH(k)(\Sigma^\infty_{\gmp}\Sigma^\infty_{S^1}U_+[i],E),\quad i\in\bb Z,\ U\in Sm/k,$$
%and the fact that the objects
%$\Sigma^\infty_{\gmp}\Sigma^\infty_{S^1}U_+[i]$ (respectively
%$\Sigma^\infty_{S^1}U_+[i]$) are compact generators of $SH^{eff}(k)$
%(respectively $SH_{S^1}(k)$).

%We can now conclude that $s_{-n}(\cc
%E)=0$ in $SH(k)$ if and only if $s_0(\cc E(n)_f)=0$ in
%$SH_{S^1}(k)$, because $\cc E(n)_f=\Omega^\infty_{\gmp}((\cc
%E(n)_f,\cc E(n+1)_f,\ldots))$ and $s_0((\cc E(n)_f,\cc
%E(n+1)_f,\ldots))\in SH^{eff}(k)$.

Since all negative slices of $\cc E$ are zero, it follows that
$f_0(\cc E)\to f_{-1}(\cc E)\to f_{-2}(\cc E)\to\cdots$ is a chain
of isomorphisms in $SH(k)$. But the canonical map
$\hocolim_{n\to+\infty} f_{-n}(\cc E)\to\cc E$ is an isomorphism in
$SH(k)$ by~\cite[4.2]{BF}, and therefore $\cc E\cong f_0(\cc E)\in
SH^{eff}(k)$ is effective as an ordinary motivic bispectrum.

Now assume the converse. Then for every $n>0$ the slice $s_{-n}(\cc
E)=0$ in $SH(k)$. We use the above arguments to conclude that
$s_0(\cc E(n)_f)=0$ in $SH_{S^1}(k)$. By Corollary~\ref{kls0cor}
$\cc E(n)_f(\wh{\Delta}^\bullet_K)=0$ in $SH$ for all finitely
generated fields $K/k$. Since $\alpha:\cc
E(n)(\wh{\Delta}^\ell_K)\to\cc E(n)_f(\wh{\Delta}^\ell_K)$ is a
stable equivalence of ordinary $S^1$-spectra for every $\ell\geq 0$
by Lemma~\ref{semiloc}, we see that $\cc
E(n)(\wh{\Delta}^\bullet_K)=0$ in $SH$. Thus $\cc E$ is effective in
the sense of Definition~\ref{dfneff}.
\end{proof}

\begin{cor}\label{freffcor}
A bispectrum $E\in SH(k)$ is effective if and only if the framed
bispectrum $\cc M_{fr}^b(E)$ is effective in the sense of
Definition~\ref{dfneff}.
\end{cor}

\begin{proof}
This follows from Theorem~\ref{thmeff} and the fact that $E\cong\cc M_{fr}^b(E)$ (see~\cite[12.4]{GP1}).
\end{proof}

\begin{cor}
The big framed motive functor $\cc M_{fr}^b:SH(k)\to SH^{fr}_{\nis}(k)$ induces an equivalence of
triangulated categories $\cc M_{fr}^b:SH^{eff}(k)\lra{\simeq} SH^{fr,\,eff}_{\nis}(k)$.
\end{cor}

\begin{proof}
This follows from Theorems~\ref{recover},~\ref{thmeff} and Corollary~\ref{freffcor}.
\end{proof}

\section{Useful lemmas}

\begin{ntt}\label{Constr_A_C_*Fr_A}
For any bispectrum $A$
such that every entry $A_{i,j}$ of $A$ are sequential
colimits of $k$-smooth simplicial schemes
denote by $C_*Fr(A)$ the bispectrum $(C_*Fr(A_{i,j}))_{i,j\geq 0}$ with obvious structure maps
   $$C_*Fr(A_{i,j})\to\uhom(S^1,C_*Fr(A_{i,j}\otimes S^1))\xrightarrow{u}\uhom(S^1,C_*Fr(A_{i+1,j}))$$
and
   $$C_*Fr(A_{i,j})\to\uhom(\gmp,C_*Fr(A_{i,j}\wedge\gmp))\xrightarrow{u}\uhom(\gmp,C_*Fr(A_{i,j+1})),$$
where $u$ refers to the structure map in $A$ in the $S^1$-$/\gmp$-direction.
Clearly, the family of maps $A_{i,j}\to C_*Fr(A_{i,j})$ form an arrow
$\zeta_A: A \to C_*Fr(A)$ of bispectra. This arrow is natural in $A$.
\end{ntt}

The proof of~\cite[12.1]{GP1} shows that the following fundamental result is true.

\begin{thm}\label{Thm_A_C_*Fr_A}
Let $A$ be a bispectrum such that every entry $A_{i,j}$ of $A$ is a sequential
colimit of $k$-smooth simplicial schemes. Then the
arrow $\zeta_A: A\to C_*Fr(A)$ is an isomorphism in $SH(k)$, functorial in $A$.
\end{thm}

\begin{cor}\label{C_*Fr_f}
Let $A,B$ be bispectra such that all entries $A_{i,j}, B_{i,j}$ of $A$ and $B$ are sequential
colimits of $k$-smooth simplicial schemes. Let $\phi: A\to B$ be a stable motivic equivalence.
Then the induced morphism
   $$C_*Fr(\phi):  C_*Fr(A)\to  C_*Fr(B)$$
is a stable motivic equivalence.
\end{cor}

\begin{proof}
This follows from the commutativity of the diagram
$$\xymatrix{A\ar[rr]^{\phi}\ar[d]_{\zeta_A} && B \ar[d]^{\zeta_B}\\
                       C_*Fr(A)\ar[rr]^{C_*Fr(\phi)} && C_*Fr(B),}$$
and Theorem \ref{Thm_A_C_*Fr_A}.
\end{proof}

\begin{cor}\label{uhom_gmp_C_*Fr_f}
Under the hypotheses of Corollary \ref{C_*Fr_f}
the natural map of bispectra
   $$\uhom(\gmpn,C_*Fr(A))\to\uhom(\gmpn,C_*Fr(B))$$
induced by the map $C_*Fr(\phi)$ is a stable motivic equivalence for all $n\geq 0$.
\end{cor}

\begin{proof}
Recall that a map of bispectra $f:X\to Y$ is a stable motivic equivalence if and only if
$\Theta^\infty_{\gmp}(f):\Theta^\infty_{\gmp}(X^{mf})\to\Theta^\infty_{\gmp}(Y^{mf})$ is a
sectionwise level equivalence, where ``$mf$" refers to the level stable motivic fibrant replacement
functor in the category of motivic $S^1$-spectra. Here $\Theta^\infty_{\gmp}$ is the standard
stabilization functor in the $\gmp$-direction.

By Corollary~\ref{C_*Fr_f} the map $C_*Fr(\phi): C_*Fr(A)\to C_*Fr(B)$ is a stable motivic equivalence.
Since both bispectra are such that in each weight stable homotopy presheaves are framed radditive stable
and $\bb A^1$-invariant, $C_*Fr(A)^{mf}=C_*Fr(A)^f$
and $C_*Fr(B)^{mf}=C_*Fr(B)^f$ by Lemma~\ref{lemprem} , where ``$f$" refers here to the level stable local fibrant replacement
functor in the category of motivic $S^1$-spectra. It follows from the sublemma of~\cite[Section~12]{GP1}
that also
$$(\uhom(\gmpn,C_*Fr(A)))^{mf}=\Omega_{\gmpn}(C_*Fr(A)^f) \ \text{and} \ (\uhom(\gmpn,C_*Fr(B)))^{mf}=\Omega_{\gmpn}(C_*Fr(B)^f).$$

We see that the map of the statement is a stable motivic equivalence if and only if the map
$\Theta^\infty_{\gmp}(\Omega_{\gmpn}(C_*Fr(A)^f))\to\Theta^\infty_{\gmp}(\Omega_{\gmpn}(C_*Fr(B)^f))$
is a sectionwise level equivalence. Since $\Theta^\infty_{\gmp}$ commutes with $\Omega_{\gmpn}$
on level motivically fibrant bispectra, the latter is equivalent to saying that
$\Omega_{\gmpn}\Theta^\infty_{\gmp}(C_*Fr(A)^f)\to\Omega_{\gmpn}\Theta^\infty_{\gmp}(C_*Fr(B)^f)$
is a sectionwise level equivalence. But the latter arrow is such, because
$\Theta^\infty_{\gmp}(C_*Fr(A)^f)\to\Theta^\infty_{\gmp}(C_*Fr(B)^f)$
is a sectionwise level equivalence.
\end{proof}

\begin{lem}\label{yyy}
Let $F$ be an $S^1$-spectrum such that every entry $F_j$ of $F$ is a sequential
colimit of $k$-smooth simplicial schemes. Then
the map of $S^1$-spectra
$$\tau_n: \uhom(\gmpn,C_*Fr(F))\to\uhom(\gmpn \wedge \gmp,C_*Fr(F\wedge\gmp))=\uhom(\gmpnl,C_*Fr(F\wedge\gmp))$$
is a stable local equivalence for all $n\geq 0$.
\end{lem}

\begin{proof}
Firstly, prove the case $n=0$.
%For any integers $j\geq 0$ the entry $F(i)_j=F_{j,i}$ is
%a sequential colimit of simplicial schemes.
The $S^1$-spectrum $F$ has a natural
filtration $F=\colim_n L_n(F)$, where
$L_n(F)$ is the spectrum
   $$(F_{0},F_{1},\ldots,F_{n},F_{n}\wedge S^1,F_{n}\wedge S^2,\ldots).$$
Then $C_*Fr(F)=C_*Fr(\colim_n L_n(F))=\colim_nC_*Fr(L_n(F))$, where
$C_*Fr(L_n(F))$ is the spectrum
   $(C_*Fr(F_{0}),C_*Fr(F_{1}),\ldots,C_*Fr(F_{n}),C_*Fr(F_{n}\otimes S^1),C_*Fr(F_{n}\otimes S^2),\ldots).$
For brevity we write it as
$$(C_*Fr(F_{0}),C_*Fr(F_{1}),\ldots,C_*Fr(F_{n-1}), M_{fr}(F_{n})[-n]),$$
where $M_{fr}(F_{n})$ is the framed motive of $F_n$ and $M_{fr}(F_{n})[-n]_r:=M_{fr}(F_{n})_{r-n}$ for $r\geq n$.
%is the frame motive of $F_{n,i}$ shifted by $-n$.
%The $S^1$-spectrum \\
Similarly,
   $$\uhom(\gmp,C_*Fr(F\wedge\gmp))=\colim_n\uhom(\gmp,C_*Fr(L_n(F)\wedge\gmp)),$$
where $\uhom(\gmp,C_*Fr(L_n(F)\wedge\gmp))$ is the $S^1$-spectrum
   $$(\uhom(\gmp,C_*Fr(F_{0}\wedge\gmp)),\ldots,\uhom(\gmp,C_*Fr(F_{n-1}\wedge\gmp)), \uhom(\gmp,M_{fr}(F_{n}\wedge\gmp)[-n])).$$
%By folklore the horizontal arrows are sectionwise stable equivalences.
%It follows from the
By the Cancellation Theorem for framed motives~\cite{AGP} the arrow
   $$M_{fr}(F_n)\to \uhom(\gmp,M_{fr}(F_n\wedge\gmp))$$
is a stable local equivalence, and hence so are the arrows
   $$C_*Fr(L_n(F))\to \uhom(\gmp,C_*Fr(L_n(F)\wedge\gmp)).$$
We conclude that the arrow $C_*Fr(F)\to \uhom(\gmp,C_*Fr(F\wedge\gmp))$
is a sequential colimit of stable local equivalences.

Next take $n>0$ and consider a commutative diagram
   $$\xymatrix{\uhom(\gmpn,C_*Fr(F))\ar[r]\ar[d]&\uhom(\gmpn,C_*Fr(F)_f)\ar[d]\\
                       \uhom(\bb G_m^{\wedge n},\uhom(\gmp,C_*Fr(F\wedge\gmp)))\ar[r]&
                       \uhom(\bb G_m^{\wedge n},\uhom(\gmp,C_*Fr(F\wedge\gmp))_f),}$$
where ``$f$" refers to local stable fibrant replacement. By the first part of the proof
the right arrow is a sectionwise level equivalence.
The horizontal arrows are stable local equivalences by the sublemma of~\cite[Section~12]{GP1}.
Our statement now follows.
\end{proof}

It is also useful to have the following important

\begin{thm}\label{eshche}
If $B$ is an $S^1$-spectrum such that every entry $B_i$ of $B$ is a sequential
colimit of $k$-smooth simplicial schemes, then the canonical morphism
$C_*Fr(B)\to \Omega^\infty_{\gmp}C_*Fr(\Sigma^\infty_{\gmp}B)$
is a stable local equivalence.
\end{thm}

\begin{proof}
It follows from Lemma~\ref{yyy} that $C_*Fr(\Sigma^\infty_{\gmp}B)\in SH_{\nis}^{fr}(k)$.
By Theorem~\ref{Thm_A_C_*Fr_A} the morphism of bispectra
   $$\Sigma^\infty_{\gmp}B\to C_*Fr(\Sigma^\infty_{\gmp}B)$$
is a stable motivic equivalence. Our theorem now follows from Corollary~\ref{sigmaomega}.
\end{proof}

\begin{lem}\label{abtheta}
Let $A,B$ be two bispectra such that every entry $A_{i,j},B_{i,j}$ of $A,B$ are sequential
colimits of $k$-smooth simplicial schemes.
\begin{itemize}
\item[$(1)$] For every weight $j\geq 0$, the $S^1$-spectrum $C_*Fr(A_{*,j})$ is $\bb A^1$-local;
\item[$(2)$] If $f:A\to B$ is a stable motivic equivalence of bispectra, then the induced map
of $S^1$-spectra $f_j:\Theta^{\infty}_{\gmp}(C_*Fr(A))(j)\to\Theta^{\infty}_{\gmp}(C_*Fr(B))(j)$
is a stable local equivalence for every weight $j\geq 0$.
\end{itemize}
\end{lem}

\begin{proof}
(1). In each weight $j$, the $S^1$-spectrum $A_{*,j}$ has a natural
filtration $A_{*,j}=\colim_n L_n(A_{*,j})$, where
$L_n(A_{*,j})$ is the spectrum
   $$A_{0,j},A_{1,j},\ldots,A_{n,j},A_{n,j}\wedge S^1,A_{n,j}\wedge S^2,\ldots$$
Then $C_*Fr(A_{*,j})=C_*Fr(\colim_n L_n(A_{*,j}))=\colim_nC_*Fr(L_n(A_{*,j}))$, where
$C_*Fr(L_n(A_{*,j})$ is the spectrum
   $$C_*Fr(A_{0,j}),C_*Fr(A_{1,j}),\ldots,C_*Fr(A_{n,j}),C_*Fr(A_{n,j}\otimes S^1),C_*Fr(A_{n,j}\otimes S^2),\ldots$$
If we take Nisnevich local replacements of all entries of the latter spectrum, we get that
   $$C_*Fr(A_{0,j})_f,C_*Fr(A_{1,j})_f,\ldots,C_*Fr(A_{n,j})_f,C_*Fr(A_{n,j}\otimes S^1)_f,C_*Fr(A_{n,j}\otimes S^2)_f,\ldots$$
is a motivically fibrant $S^1$-spectrum starting from level $n+1$ by~\cite[7.5]{GP1}. So
each $C_*Fr(L_n(A_{*,j}))$ is $\bb A^1$-local, and hence so is $C_*Fr(A_{*,j})$
(we use here the fact that a sequential colimit of motivically fibrant spectra is
motivically fibrant in the stable projective model structure).

(2). By Corollary \ref{C_*Fr_f} the induced map $C_*Fr(f):C_*Fr(A)\to C_*Fr(B)$ is a
stable motivic equivalence of bispectra.
%It is shown similar to~\cite[12.1]{GP1} that $A\to C_*Fr(A)$ and $B\to C_*Fr(B)$ are
%stable motivic equivalences of bispectra, and hence so is the induced map $f:C_*Fr(A)\to C_*Fr(B)$.
Let ``$f$" refer to local stable fibrant replacement.
Since in each weight $j\geq 0$ the $S^1$-spectra $C_*Fr(A(j))_f,C_*Fr(B(j))_f$ are motivically
fibrant, the induced map $f:\Theta^{\infty}_{\gmp}(C_*Fr(A)_f)\to\Theta^{\infty}_{\gmp}(C_*Fr(B)_f)$
is a sectionwise level equivalence with $C_*Fr(A)_f=(C_*Fr(A(0))_f,C_*Fr(A(1))_f,\ldots)$,
$C_*Fr(B)_f=(C_*Fr(B(0))_f,C_*Fr(B(1))_f,\ldots)$ respectively.

Consider a commutative diagram
   $$\xymatrix{\Theta^{\infty}_{\gmp}(C_*Fr(A))(j)\ar[r]^{f}\ar[d]
       &\Theta^{\infty}_{\gmp}(C_*Fr(B))(j)\ar[d]\\
       \Theta^{\infty}_{\gmp}(C_*Fr(A)_f)(j)\ar[r]^{f}
       &\Theta^{\infty}_{\gmp}(C_*Fr(B)_f)(j)}$$
It follows from the sublemma of~\cite[Section~12]{GP1} that the vertical arrows are
stable local equivalences. Since the bottom arrow is a sectionwise level equivalence,
we get that the top arrow is a stable local equivalence, as required.
\end{proof}

We finish the section by the following useful

\begin{lem}
If $A$ is a bispectrum such that every entry $A_{i,j}$ of $A$ is a sequential
colimit of $k$-smooth simplicial schemes, then $\Theta^{\infty}_{\gmp}(C_*Fr(A))\in SH_{\nis}^{fr}(k)$.
\end{lem}

\section{The category of spectral functors}
In this section we introduce the category of spectral functors, which will be used
to give other models for $SH(k)$ and $SH^{eff}(k)$. Recall that the category
of pointed Nisnevich sheaves $Shv_\bullet(Sm/k)$ is closed symmetric monoidal.
For brevity, we shall often write $[F,G]$ to denote $\uhom(F,G)$.

\begin{lem}\label{cfr0}
For any $U,V,X\in Sm/k$ one has
   $$[U_+,V_+](X)=\Hom_{Shv_\bullet(Sm/k)}((U\times X)_+,V_+)=Fr_0(U\times X,V),$$
where  and $Fr_0(U\times X,V)$ is the pointed set of framed correspondences of level 0 in the sense
of Voevodsky~\cite{Voe2}.
\end{lem}

\begin{proof}
One has
   $$[U_+,V_+](X)={\Hom}_{Shv_\bullet(Sm/k)}(U_+\wedge X_+,V_+)=\Hom_{Shv_\bullet(Sm/k)}((U\times X)_+,V_+).$$
The fact that $\Hom_{Shv_\bullet(Sm/k)}((U\times
X)_+,V_+)=Fr_0(U\times X,V)$ is shown in~\cite{Voe2} (see~\cite{GP1}
as well).
\end{proof}

\begin{rem}
It is worth mentioning that $Fr_0(-,V)$, $V\in Sm/k$, is the
Nisnevich sheaf associated to the presheaf $U\in
Sm/k\mapsto\Hom_{Sm/k}(U,V)\sqcup pt$. Whenever $U$ is connected, we
have that $Fr_0(U,V)=\Hom_{Sm/k}(U,V)\sqcup pt$.
\end{rem}

\begin{dfn}
(1) We denote by $\Fr$ the category enriched over
$Shv_\bullet(Sm/k)$ whose objects are those of $Sm/k$ and
$Shv_\bullet(Sm/k)$-objects given by $Fr_0(U\times-,V)$, $U,V\in
Sm/k$. It is indeed a $Shv_\bullet(Sm/k)$-category by
Lemma~\ref{cfr0}. Note that $\Fr$ is $Shv_\bullet(Sm/k)$-symmetric
monoidal with $U\otimes V:=U\times V$. Clearly, its underlying
category is $Fr_0(k)$, i.e. the category of smooth $k$-schemes and
framed correspondences of level 0.

(2) Denote by $\Fun(\Fr):=[\Fr,Shv_\bullet(Sm/k)]$ the
$Shv_\bullet(Sm/k)$-category of enriched functors from $\Fr$ to
$Shv_\bullet(Sm/k)$. It is closed symmetric monoidal by Day's
Theorem~\cite{Day}. Its monoidal unit $[pt,-]$ is represented by
$pt\in\Fr$. By~\cite[2.4]{DRO} $\Fun(\Fr)$ is tensored and
cotensored over $Shv_\bullet(Sm/k)$.

(3) Consider the category $\Delta^{\op}\Fun(\Fr)$ of
simplicial objects in $\Fun(\Fr)$. By~\cite[2.4]{DRO}
$\Delta^{\op}\Fun(\Fr)$ is tensored and cotensored over pointed
motivic spaces $\mathbb M_\bullet:=\Delta^{\op}Shv_\bullet(Sm/k)$.
Note that $\Fr$ is also enriched over $\mathbb M_\bullet$ in a canonical way.
Denote by $Sp_{S^1}[\Fr]$ the category of $S^1$-spectra associated
with the category $\Delta^{\op}\Fun(\Fr)$ and call it the {\it
category of spectral functors}. By definition, the {\it underlying
$S^1$-spectrum\/} of a spectral functor $\cc X\in Sp_{S^1}[\Fr]$ is
its evaluation $\cc X(pt)$ at $pt$.
\end{dfn}

\begin{ex}
Let $Fr_n(U,V)$, $n\geq 0$, be the set of framed correspondences of
level $n$ pointed at the empty correspondence. Then $Fr_n(-,-)$ has
an action of the category $Fr_0(k)$. The assignment
   $$\cc Fr_n(-,X):V\in\Fr\mapsto Fr_n(-,X\times V)\in Shv_\bullet(Sm/k),\quad X\in Sm/k,$$
together with the maps
   $$[U_+,V_+]\mapsto \underline{\Hom}(Fr_n(-,X\times U), Fr_n(-,X\times V)),$$
induced by the action of $Fr_0(k)$ on $Fr_n(U,V)$-s, gives an object
of $\Fun(\Fr)$ denoted by $\cc Fr_n(-,X)$. Stabilizing over $n$, we
get an object of $\Fun(\Fr)$
   $$\cc Fr(-,X):=\colim(\cc Fr_0(-,X)\xrightarrow\sigma\cc Fr_1(-,X)\xrightarrow\sigma\cdots).$$

We can equally define $\cc Fr(F\wedge-,X):=\uhom(F,\cc Fr(-,X))\in\Fun(\Fr)$ for any pointed sheaf
$F\in Shv_\bullet(Sm/k)$.
If we set $C_*\cc Fr(-,X):=\cc
Fr(\Delta^\bullet\times-,X)$, we get an object of
$\Delta^{\op}\Fun(\Fr)$.

Let $\Delta[\bullet]$ be the standard cosimplicial simpicial set $n\longmapsto\Delta[n]$. If there is no
likelihood of confusion, we sometimes regard it as a cosimplicial smooth scheme, where each
$\Delta[n]$ is regarded as the disjoint union $\bigsqcup_{\Delta[n]}\spec(k)$.
Recall that the simplicial function object between pointed motivic spaces $A$ and $B$ is given by
   $${\bf S}_{\bullet}(A,B)=\Hom_{\bb M_\bullet}(A\wedge\Delta[\bullet]_{+},B)=\Hom_{\bb M_\bullet}(A,B(\Delta[\bullet]\times-)).$$
In turn, the $\bb M_\bullet$-object of morphisms between $A$ and $B$ is defined by
   $$X\in Sm/k\longmapsto{\bf S}_{\bullet}(A,B(X\times-))
       =\Hom_{\bb M_\bullet}(A,B(X\times\Delta[\bullet]\times-)).$$

In Section~\ref{frbisp} we defined 
the category of framed motivic spaces $\bb M^{fr}_\bullet$ and 
the canonically induced faithful functor
$\iota:\bb M^{fr}_\bullet\rightarrow\bb M_\bullet$ obtained from
$i:{Sm}/k\to{Fr}_{+}(k)$. 
Moreover, the definition of the pairing
${Fr}_{0}(k)\times{Fr}_{+}(k)\xrightarrow{\otimes}{Fr}_{+}(k)$ was given in that section.
It takes $(X,Y)$ to $X\times Y$ and $(f,\alpha)$ to $f\times \alpha$.

We are going to define a natural enrichment 
of the category of framed motivic spaces $\bb M^{fr}_\bullet$ over $\bb M_\bullet$.
%%%Let ${Fr}_+(k)$ be the category of framed correspondences as in ~\cite[Section~2]{GP1}. Let
%%%${Pre}^{fr}(k)$ be the category of framed presheaves, that is the category of presheaves of sets
%%%on ${Fr}_{+}(k)$.
%%%Let $i:{Sm}/k\to{Fr}_{+}(k)$ be the canonical functor.
%%%Recall from~\cite[Section~2]{GP1} that a framed Nisnevich sheaf on ${Sm}/k$ is a framed presheaf
%%%such that its restriction to ${Sm}/k$
%%%via the functor $i$
%%%is a Nisnevich sheaf. Let $Shv_\bullet^{fr}(k)$ denote the category
%%%of pointed framed Nisnevich sheaves.
%%%The morphisms in this category are just morphisms of pointed framed presheaves.

%%%Following~\cite[Section~6]{Voe2} there is a natural pairing
%%%${Fr}_{0}(k)\times{Fr}_{+}(k)\xrightarrow{\otimes}{Fr}_{+}(k)$
%%%taking $(X,Y)$ to $X\times Y$ and $(f,\alpha)$ to $f\times \alpha$. In what follows
%%%this pairing will be used systematically without referring to it.
%%%We also use it in the natural enrichment of $\bb M^{fr}_\bullet$ over $\bb M_\bullet$.

First, we can associate a framed Nisnevich sheaf $\cc F(X\times -)$ to
every framed Nisnevich sheaf $\cc F$ and every $X\in{Fr}_{0}(k)$. In detail,
given $\alpha\in{Fr}_n(U',U)$ put
$\alpha^*: \cc F(X\times U)\to \cc F(X\times U')$
to be $(\id_X\times \alpha)^*$.
If $\cc F$ is a pointed framed Nisnevich sheaf then the framed Nisnevich sheaf
$\cc F(X\times -)$ is pointed also.

Second, every morphism $f: X'\to X$ in ${Fr}_{0}(k)$ induces a morphism
of framed sheaves $f^*: \cc F(X\times -)\to \cc F(X'\times -)$. Namely, if $U\in {Fr}_+(k)$ one sets
$f^*: \cc F(X\times U)\to \cc F(X'\times U)$ to be $(f\times \id_U)^*$.
If $\cc F$ is a pointed framed Nisnevich sheaf, then the morphism of framed sheaves
$f^*: \cc F(X\times -)\to \cc F(X'\times -)$
is a morphism of pointed framed Nisnevich sheaves.

Finally, $\bb M^{fr}_\bullet$ is naturally enriched over $\bb M_\bullet$. Namely,
$$
{\bb M_\bullet}(A,B)(X):=\Hom_{\bb M^{fr}_\bullet}(A,B(X\times\Delta[\bullet]\times-)),
\quad A,B\in\bb M^{fr}_\bullet, \ X\in {Sm}/k.
$$
The enriched composition in $\bb M^{fr}_\bullet$ is inherited from the enriched composition in $\bb M_\bullet$.

\begin{dfn}\label{frsp}
A {\it framed motivic enriched functor\/} is an $\bb M_\bullet$-enriched functor $\cc X:\cc Fr_0(k)\to\bb M^{fr}_\bullet$
between the $\bb M_\bullet$-enriched categories $\cc Fr_0(k)$ and $\bb M^{fr}_\bullet$.
A typical example of a framed motivic enriched functor is given by $C_*\cc Fr(-,X)$.

Denote by $\Fun^{fr}(\Fr)$ the category of the framed motivic enriched functors and
$\bb M_\bullet$-natural transformations between them.
\end{dfn}

\begin{rem}
The underlying category of the $\bb M_\bullet$-category $\cc Fr_0(k)$
equals $Fr_0(k)$. Every $\bb M_\bullet$-enriched functor
$\cc X:\cc Fr_0(k)\to\bb M_\bullet$ gives rise to a functor
$\cc X\colon Fr_0(k)\rightarrow\bb M_\bullet$ denoted by the same letter.

Unravelling the previous definition, a framed motivic enriched functor is equivalent to giving the following data:
\begin{itemize}
\item[$\diamond$] an $\bb M_\bullet$-functor $\cc X:\cc Fr_0(k)\to\bb M_\bullet$;
\item[$\diamond$] a functor $\cc X':Fr_0(k)\to\bb M^{fr}_\bullet$;
\item[$\diamond$]
the induced functor $\cc X:Fr_0(k)\to\bb M_\bullet$
equals the composite functor $Fr_0(k)\xrightarrow{\cc X'}\bb M^{fr}_\bullet\xrightarrow{\iota}\bb M_\bullet$
such that the canonical morphism
$$
[U,V](Y)
\longrightarrow
\Hom_{\bb M_\bullet}(\cc X(U),\cc X(V)(Y\times-))
$$
factors through $\Hom_{\bb M^{fr}_\bullet}(\cc X'(U),\cc X'(V)(Y\times-))$ for all
$U,V,Y\in Fr_0(k)$.
\end{itemize}
\end{rem}

Voevodsky~\cite[Section~3]{VoeICM}
defined a realization functor from simplicial sets to Nisnevich
sheaves $|-|:sSets\to Shv_{\nis}(Sm/k)$ such that
$|\Delta[n]|=\Delta^n_k$, where $\Delta[n]$ is the standard
$n$-simplex. Under this notation the cosimpicial scheme
$\Delta^\bullet_k$ equals $|\Delta[\bullet]|$. For every $\ell\geq 0$
denote by $sd^\ell\Delta^\bullet_k$ the cosimplicial Nisnevich sheaf
$|sd^\ell\Delta[\bullet]|$. Under this notation we then have a
canonical isomorphism in $\Delta^{\op}\Fun(\Fr)$
   $$Ex^\ell(C_*\cc Fr(-, X))=\cc Fr(|sd^\ell\Delta[\bullet]|_+\wedge-,X).$$
It follows that $Ex^\ell(C_*\cc Fr(\cc X))\in\Delta^{\op}\Fun(\Fr)$ as well as
   $$Ex^\infty(C_*\cc Fr(\cc X)):=\colim_\ell Ex^\ell(C_*\cc Fr(\cc X))\in\Delta^{\op}\Fun(\Fr).$$
Then $Ex^\infty(C_*\cc Fr(-,X))$ is sectionwise fibrant in the first argument
and the map $C_*\cc Fr(-, X)\to
Ex^\infty(C_*\cc Fr(-, X))$ is a sectionwise weak equivalence in the first argument.

Likewise, the Segal spectrum $\cc
M_{fr}(X):=Ex^\infty(C_*\cc Fr(-,X\otimes\bb S))$, where $\bb S$ is the
ordinary sphere spectrum, gives a spectral functor. Observe that the
underlying $S^1$-spectrum of $\cc M_{fr}(X)$ is the framed motive
$M_{fr}(X)$ of $X$ in the sense of~\cite{GP1}. We shall refer to
$\cc M_{fr}(X)$ as the {\it framed motive spectral functor of $X$}.
\end{ex}

There is a natural evaluation functor
   \begin{equation}\label{eval}
    ev_{\bb G_m}:Sp_{S^1}[\Fr]\to Sp_{S^1,\bb G_m}(k),
   \end{equation}
where $Sp_{S^1,\bb G_m}(k)$ is the category of
$(S^1,\gmp)$-bispectra in $\bb M_\bullet$, defined as follows. For
every $U,V\in Sm/k$ and $\cc X\in\Fun(\Fr)$ we have natural
morphisms
   $$V_+\to[U_+,(U\times V)_+]\to\underline{\Hom}_{Shv_\bullet(Sm/k)}(\cc X(U),\cc X(U\times V))$$
inducing a morphism $\cc X(U)\wedge V_+\to\cc X(U\times V)$ in
$Shv_\bullet(Sm/k)$.
%In particular, we have a morphism $\cc
%X(U)\wedge (\bb G_m^{\sqcup_{<\infty}})_+\to\cc X(U\times \bb
%G_m^{\sqcup_{<\infty}})$ in $Shv_\bullet(Sm/k)$, where $\bb
%G_m^{\sqcup_{<\infty}}$ is the disjoint union of finitely many
%copies of $\bb G_m$ (recall that the simplices of $\gmpn$ are of the
%form $\bb G_m^{\sqcup_{<\infty}}$).
Given $\cc X\in\Delta^{\op}\Fun(\Fr)$, we define $\cc X(\gmpn)\wedge\gmp\to\cc X(\bb
G_m^{\wedge n+1})$ as the geometric realization of morphisms
$[k\in\Delta^{\op}\mapsto\{\cc X((\bb G_m^{\wedge
n})_k)\wedge(\gmp)_k\to\cc X((\bb G_m^{\wedge (n+1)})_k)\}]$.
The latter morphisms are now easily extended to spectral functors
yielding the desired evaluation functor $ev_{\bb G_m}$. One also has
the induced morphism
   $$\cc X(\gmpn)\to\uhom(\gmp,\cc X((\bb G_m^{\wedge (n+1)})).$$
%%%It is worth noticing 
The claim of Section~\ref{frbisp} implies that if 
$\cc X\in\Fun^{fr}[\Fr]$, then $\cc X(\gmpn)\to\uhom(\gmp,\cc X((\bb G_m^{\wedge (n+1)}))$
is a morphism in $\bb M^{fr}_\bullet$.

\section{Recovering $SH(k)$ from framed spectral functors}\label{recspfun}

In Section~\ref{frbisp} we recovered the stable motivic homotopy
category $SH(k)$ as the category of framed bispectra. In this
section we give another method of reconstructing $SH(k)$ from framed
spectral functors.

\begin{dfn}\label{frspfundfn}
(1) We say that a graded object $\cc X=(\cc X_0,\cc X_1,\ldots)$ in $\Fun^{fr}[\Fr]$ of framed motivic enriched functors
in the sense of Definition~\ref{frsp} is a {\it framed spectral functor\/} if the following properties are satisfied:

\begin{itemize}
\item[$\diamond$] for each $i\geq 0$, there is a morphism $\cc X_i\to\uhom(S^1,\cc X_{i+1})$ in $\Fun^{fr}[\Fr]$. In
particular, $\cc X$ is an $S^1$-spectrum of framed motivic enriched functors;

\item[$\diamond$] (Cancellation Theorem) for every $U\in Sm/k$, the canonical
morphism between $S^1$-spectra of framed motivic spaces
   $$\cc X(\gmpn)\to\uhom(\gmp,\cc X((\bb G_m^{\wedge (n+1)}))$$
is a stable local equivalence for any $n\geq 0$. In particular, the
bispectrum, defined as in~\eqref{eval},
   $$ev_{\bb G_m}(\cc X(-\times U))=(\cc X(U),\cc X(\gmp\times U),\cc X(\bb G_m^{\wedge 2}\times U),\ldots)$$
is a framed bispectrum in the sense of
Definition~\ref{deffrsp}, i.e. $ev_{\bb G_m}(\cc X(-\times U))\in
SH_{\nis}^{fr}(k)$;

\item[$\diamond$]  for every $U\in Sm/k$ the canonical map $pr:\cc X(\bb A^1\times U)\to\cc X(U)$
is a stable local equivalence of $S^1$-spectra of framed motivic spaces;

\item[$\diamond$] it is {\it Nisnevich excisive\/} in the sense that
for every elementary Nisnevich square in $Sm/k$
   $$\xymatrix{U'\ar[r]\ar[d]&Y'\ar[d]\\
                       U\ar[r]&Y}$$
the square of the $S^1$-spectra of framed motivic spaces
   $$\xymatrix{\cc X(U')\ar[r]\ar[d]&\cc X(Y')\ar[d]\\
                       \cc X(U)\ar[r]&\cc X(Y)}$$
is locally homotopy cartesian (=homotopy cocartesian in the stable structure).
We also require $\cc X(\emptyset)=*$.
\end{itemize}

(2) We define a category $SH_{S^1}^{fr}[\Fr]$ as follows. Its
objects are the framed spectral functors and morphisms are defined as
   $$SH_{S^1}^{fr}[\Fr](\cc X,\cc Y):=SH_{\nis}^{fr}(k)(ev_{\bb
      G_m}(\cc X), ev_{\bb G_m}(\cc Y)).$$
We call $SH_{S^1}^{fr}[\Fr]$ the {\it category of framed spectral functors}.
\end{dfn}

A basic example of a framed spectral functor is the framed motive spectral functor
$\cc M_{fr}(X)$ of $X\in Sm/k$. If we apply $\uhom(\gmpn,-)$ to $\cc M_{fr}(X)$, $n\geq 0$, we get
another framed spectral functor, denoted by $\Omega_{\gmpn}\cc M_{fr}(X)$.
Explicitly, $\Omega_{\gmpn}\cc M_{fr}(X)=C_*\cc Fr(\gmpn\times-,X\otimes\bb S)$.
It is important to note that, by definition, the category
$SH_{S^1}^{fr}[\Fr]$ has nothing to do with any sort of motivic
equivalences. It is purely of local nature.

\begin{dfn}\label{upsilon1}
Let $\cc F$ be a symmetric bispectrum. Let $\cc F[n]$ be the usual $n$th shift of $\cc F$ in the $\gmp$-direction.
For each $i\geq 0$ let $\cc F(i)$ and $\cc F[1](i)$ be the $i$th weights of $\cc F$ and $\cc F[1]$ respectively
(see Notation \ref{weight}).
Recall that there is a map of {\it symmetric\/} bispectra
$$\upsilon_{\cc F}: \gmp\wedge\cc F\to\cc F[1]$$
defined as follows: for each weight $i\geq 0$ it is the composition of maps of $S^1$-spectra
   $$\gmp\wedge \cc F(i)\xrightarrow{tw} \cc F(i)\wedge\gmp\xrightarrow{u}\cc F(i+1)\xrightarrow{\chi_{i,1}}\cc F(1+i)=\cc F[1](i),$$
where $\chi_{i,1}$ is the shuffle permutation.
Sometimes we write $\upsilon$ to denote $\upsilon_{\cc F}$ dropping $\cc F$ from notation.
Notice that $\upsilon_{\cc F}$ cannot be defined for nonsymmetric bispectra.
\end{dfn}

The main result of this section is as follows.

\begin{thm}\label{spfuncat}
The category of spectral framed functors $SH_{S^1}^{fr}[\Fr]$ is
compactly generated triangulated with compact generators
$\{\Omega_{\gmpn}\cc M_{fr}(X)\mid X\in Sm/k,n\geq 0\}$ and the shift functor $\cc X[1]=\cc X(-\otimes S^1)$.
Moreover, the composite functor
   $$F\circ ev_{\bb G_m}:SH_{S^1}^{fr}[\Fr]\to SH(k)$$
is an equivalence of triangulated categories, where
$F:SH_{\nis}^{fr}(k)\to SH(k)$ is the triangulated equivalence of
Theorem~\ref{recover}.
\end{thm}

\begin{proof}
The fact that $F\circ ev_{\bb G_m}$ is fully faithful follows from Definition~\ref{frspfundfn}(2)
and Theorem~\ref{recover}. We need to show that every bispectrum $E\in SH(k)$ is isomorphic to
the evaluation bispectrum of a spectral framed functor. Without loss of generality
we may assume that $E$ is a symmetric $(S^1,\gmp)$-bispectrum which is also fibrant
in the stable projective motivic model structure of symmetric bispectra. Then it is also motivically fibrant
as a (non-symmetric) bispectrum. Let $\alpha:E^c\to E$ be the cofibrant replacement of $E$ in the category
of symmetric bispectra. Then $\alpha$ is an isomorphism in $SH(k)$ and every entry of $E^c$
is a sequential colimit of $k$-smooth simplicial schemes.

Let $R^n_{\gmp} C_*Fr(E^c)$, $n\geq 0$, be the bispectrum $\uhom(\gmpn,C_*Fr(E^c[n]))$,
where $E^c[n]$ is the usual $n$th shift of $E^c$ in the $\gmp$-direction. In more detail, for each
weight $i\geq 0$ it equals the $S^1$-spectrum
   $$R^n_{\gmp} C_*Fr(E^c)(i):=\uhom(\gmpn,C_*Fr(E^c(n+i))),$$
where $E^c(n+i)$ is the $(n+i)$th weight of the bispectrim $E^c$. Each bonding map equals the composition
   \begin{multline*}
    \uhom(\gmpn,C_*Fr(E^c(n+i)))\wedge\gmp\to\uhom(\gmpn,C_*Fr(E^c(n+i)\wedge\gmp))\xrightarrow u\\
    \to\uhom(\gmpn,C_*Fr(E^c(n+i+1))).
   \end{multline*}
It is useful to regard the symmetric bispectrum $E^c=(E^c(0),E^c(1),\ldots)$ as a symmetric $\gmp$-spectrum
in the category of symmetric motivic $S^1$-spectra. Then there is a map of bispectra
   \begin{equation}\label{rn}
    R^n_{\gmp} C_*Fr(E^c)\to R^{n+1}_{\gmp} C_*Fr(E^c)
   \end{equation}
defined at each weight $i$ as the composition
   \begin{gather*}\label{rni}
    \uhom(\gmpn,C_*Fr(E^c(n+i)))\to\uhom(\gmpn,\uhom(\gmp,C_*Fr(E^c(n+i)\wedge\gmp)))\xrightarrow{u}\\
     \uhom(\bb G_m^{\wedge n+1},C_*Fr(E^c(n+i+1)))\xrightarrow{\chi_{i,1}}\uhom(\bb G_m^{\wedge n+1},C_*Fr(E^c(n+1+i))),
   \end{gather*}
where $\chi_{i,1}$ is the shuffle permutation in $\Sigma_{n+i+1}$ permuting the last element with
preceding $i$ elements and preserving the first $n$ elements. An equivalent definition is given by
the following composition of bispectra
   \begin{gather*}
    \uhom(\gmpn,C_*Fr(E^c[n]))\xrightarrow{\alpha_n} \uhom(\gmpn,\uhom(\gmp,C_*Fr(E^c[n]\wedge\gmp)))\xrightarrow{tw}\\
     \uhom(\bb G_m^{\wedge n+1},C_*Fr(\gmp\wedge E^c[n]))\xrightarrow{\beta_n}\uhom(\bb G_m^{\wedge n+1},C_*Fr(E^c[n+1])),
   \end{gather*}
where the arrow $\beta_n$ is induced by the canonical map of {\it symmetric\/} bispectra
$$\upsilon_n=\upsilon_{E^c[n]}: \gmp\wedge E^c[n]\to E^c[n][1]=E^c[n+1].$$
%Recall that
Following Definition~\ref{upsilon1} in each weight $i\geq 0$ it is the composition of maps of $S^1$-spectra
   $$\gmp\wedge E^c[n](i)\xrightarrow{tw}E^c[n](i)\wedge\gmp\xrightarrow{u}E^c[n](i+1)\xrightarrow{\chi_{i,1}}E^c[n](1+i).$$
Recall that $\upsilon_n$ cannot be defined for nonsymmetric bispectra.

%Set,
%   $$R^\infty_{\gmp} C_*Fr(E^c):=\colim(C_*Fr(E^c)\to R^1_{\gmp} C_*Fr(E^c)\to R^2_{\gmp} C_*Fr(E^c)\to\cdots).$$
%We want to show that
{\it Claim 1.} The map~\eqref{rn} is a stable motivic equivalence of bispectra.
%\vskip 0,2cm

To prove this claim, we need the following

%In this case the composition
%   $$E^c\xrightarrow{\zeta_{E^c}} C_*Fr(E^c)\to R^\infty_{\gmp} C_*Fr(E^c)$$
%is a stable motivic equivalence, because the left arrow is a stable motivic equivalence
%by Theorem \ref{Thm_A_C_*Fr_A}
%as well as so is the right arrow (we use the fact that a sequential colimit of stable motivic equivalences is
%a stable motivic equivalence). To this end, we prove several lemmas below.

\begin{lem}\label{upsilon}
Suppose $\cc F$ is a motivically fibrant symmetric bispectrum. Then the morphism
   $$\upsilon:\gmp\wedge\cc F\to\cc F[1]$$
from Definition \ref{upsilon1} is a stable
motivic equivalence of ordinary nonsymmetric bispectra. More generally, the arrow
$$\upsilon_{A\wedge\cc F}: \gmp\wedge(A\wedge\cc F)\to (A\wedge\cc F)[1]$$
is a stable
motivic equivalence of ordinary nonsymmetric bispectra for any projectively cofibrant motivic space $A\in\bb M_\bullet$.
\end{lem}

\begin{proof}
Consider a commutative diagram
   $$\xymatrix{\cc F\ar[dr]^\sim\ar[r]&\uhom(\gmp,\gmp\wedge\cc F)\ar[r]\ar[d]&\Omega_{\gmp}(\cc F[1])\ar[d]^\sim\\
                       &\Omega_{\gmp}((\gmp\wedge\cc F)^{nf})\ar[r]&\Omega_{\gmp}(\cc F[1]^{nf}).}$$
Here $nf$ refers to the fibrant replacement functor in the category of ordinary nonsymmetric bispectra.
The left slanted arrow and the right vertical arrow are stable
motivic equivalences of ordinary nonsymmetric bispectra. The upper composite arrow is a
stable motivic equivalence between symmetric spectra. Since both spectra are motivically fibrant,
it follows that this composite arrow is a naive sectionwise and levelwise weak equivalence.
We see that the bottom arrow is a stable motivic equivalence of bispectra.
Since $\Omega_{\gmp}$ reflects stable motivic equivalences of bispectra,
$(\gmp\wedge\cc F)^{nf}\to\cc F[1]^{nf}$ is a stable motivic equivalence of bispectra.
The latter implies $\gmp\wedge\cc F\to\cc F[1]$ is a stable motivic equivalence of bispectra as well.

Now let $A$ be a projectively cofibrant motivic space. Then factor $\upsilon:\gmp\wedge\cc F\to\cc F[1]$
as the composition $\gmp\wedge\cc F\to Z\to\cc F[1]$ of a projective stable trivial cofibration
followed by a fibration of nonsymmetric bispectra. It follows that
$\gmp\wedge (A\wedge\cc F)\to A\wedge Z$ is a stable
trivial cofibration. Since $\cc F$ is fibrant as a nonsymmetric bispectrum, then so is $Z$.
By assumption, $\upsilon$ is a stable motivic equivalence. Therefore $Z\to\cc F[1]$ is a level
equivalence. Regarding the entries of $Z$ and $\cc F[1]$ as injective cofibrant motivic spaces,
we conclude that $A\wedge Z\to (A\wedge\cc F)[1]$ is a level equivalence.
\end{proof}

%By the construction from Notation
%\ref{Constr_A_C_*Fr_A}
%there is a commutative diagram of bispectra
%   $$\xymatrix{\gmp\wedge E^c[n]\ar[r]^{\upsilon}\ar[d]&E^c[n+1]\ar[d]\\
%                       C_*Fr(\gmp\wedge E^c[n])\ar[r]& C_*Fr(E^c[n+1]),}$$
%where the bottom arrow is induced by the arrow $\upsilon$.
Now let us prove Claim 1. The arrow
$\gmp\wedge E^c[n]\xrightarrow{\upsilon_n} E^c[n+1]$
is a stable motivic equivalence by the preceding lemma.
Every entry of $\gmp\wedge E^c[n]$ and $E^c[n+1]$
is a sequential colimit of $k$-smooth simplicial schemes.
Thus by Corollary
\ref{C_*Fr_f}
the arrow
$C_*Fr(\upsilon_n): C_*Fr(\gmp\wedge E^c[n])\to C_*Fr(E^c[n+1])$
%Both vertical arrows are stable motivic equivalence by
%Theorem \ref{Thm_A_C_*Fr_A}.
%Applying the arguments from the proof of
%\cite[Thm.12.1]{GP1}
%as well as
%so
%We conclude that so are both vertical arrows.
%Hence the bottom arrow
is a stable motivic equivalence.
By Corollary \ref{uhom_gmp_C_*Fr_f}
the map of bispectra
$$\beta_{n}: \uhom(\gmpnl,C_*Fr(\gmp\wedge E^c[n]))\to\uhom(\gmpnl,C_*Fr(E^c[n+1]))$$
is a stable motivic equivalence for all $n\geq 0$.
By Lemma \ref{yyy}
the map of $S^1$-spectra
$$\alpha_n: \uhom(\gmpn,C_*Fr(E^c[n]))\to\uhom(\bb G_m^{\wedge n+1},C_*Fr(E^c[n]\wedge\gmp))$$
is a level stable local equivalence for all $n\geq 0$.
%As a consequence of the previous two lemmas
The maps $\alpha_n$ and $\beta_n$ are stable motivic equivalences, and hence so is
the map~\eqref{rn} of bispectra and Claim 1 follows.

We continue the proof of the theorem by setting
   $$R^\infty_{\gmp} C_*Fr(E^c):=\colim(C_*Fr(E^c)\to R^1_{\gmp} C_*Fr(E^c)\to R^2_{\gmp} C_*Fr(E^c)\to\cdots).$$
{\it Claim 2.} The composition
$$E^c\xrightarrow{\zeta_{E^c}} C_*Fr(E^c)\to R^\infty_{\gmp} C_*Fr(E^c)$$
is a stable motivic equivalence.
%\vskip 0,2cm
%This Claim is true because

The left arrow is a stable motivic equivalence
by Theorem \ref{Thm_A_C_*Fr_A}.
The right arrow a stable motivic equivalence by
Claim 1 (we use the fact that a sequential colimit of stable motivic equivalences is
a stable motivic equivalence).

To complete the proof of the theorem, we need to construct a framed spectral functor such that its evaluation bispectrum is
equivalent to $R^\infty_{\gmp}C_*Fr(E^c)$.

For any $n\geq 0$ we define a spectral functor $\bb GC_*Fr^{E}[n]$ by
   $$X\in\cc Fr_0(k)\mapsto\uhom(\gmpn,C_*Fr(X_+\wedge E^c(n))).$$
Recall that
$ev_{\bb G_m}(\bb GC_*Fr^{E}[n])(i)=\bb GC_*Fr^{\bb G_m^{\wedge i}\wedge E}[n](pt)=\uhom(\gmpn,C_*Fr(\bb G_m^{\wedge i}\wedge E^c(n)))$.
The bonding maps $ev_{\bb G_m}(\bb GC_*Fr^{E}[n])(i)\wedge\gmp\to ev_{\bb G_m}(\bb GC_*Fr^{E}[n])(i+1)$
are defined as the composite maps
   \begin{gather*}
       \uhom(\gmpn,C_*Fr(\bb G_m^{\wedge i}\wedge E^c(n)))\wedge\gmp\to
       \uhom(\gmpn,C_*Fr(\bb G_m^{\wedge i}\wedge E^c(n)\wedge\gmp))\\
       \xrightarrow{tw}\uhom(\gmpn,C_*Fr(\bb G_m^{\wedge i}\wedge\gmp\wedge E^c(n))).
    \end{gather*}
It follows now from Lemma~\ref{yyy} that $ev_{\bb G_m}(\bb GC_*Fr^{E}[n])\in SH^{fr}_{\nis}(k)$.
%Thus $\bb GC_*Fr^{E}[n]$ is a framed spectral functor.

There is a natural morphism of spectral functors $\bb GC_*Fr^{E}[n]\to \bb GC_*Fr^{E}[n+1]$
defined at every $X\in\cc Fr_0(k)$ by the composition
   \begin{gather*}
    \uhom(\gmpn,C_*Fr(X_+\wedge E^c(n)))\to\uhom(\bb G_m^{\wedge n+1},C_*Fr(X_+\wedge E^c(n)\wedge\gmp))\\
       \xrightarrow u\uhom(\bb G_m^{\wedge n+1},C_*Fr(X_+\wedge E^c(n+1))).
   \end{gather*}

\begin{dfn}\label{cc_M_fr_E}
Set $\cc M_{fr}^{E}:=\colim(\bb GC_*Fr^{E}[0]\to\bb GC_*Fr^{E}[1]\to\cdots)$.
By construction, $ev_{\bb G_m}(\cc M_{fr}^{E})=\colim_n ev_{\bb G_m}(\bb GC_*Fr^{E}[n])$.
Each $\bb GC_*Fr^{E}[n]$ is a framed spectral functor, and hence so is
$\cc M_{fr}^{E}$.
\end{dfn}

Stabilization maps $ev_{\bb G_m}(\bb GC_*Fr^{E}[n])\to ev_{\bb G_m}(\bb GC_*Fr^{E}[n+1])$ are given by the compositions
   \begin{gather*}
       \uhom(\gmpn,C_*Fr(\bb G_m^{\wedge i}\wedge E^c(n)))\to
       \uhom(\gmpn\wedge\gmp,C_*Fr(\bb G_m^{\wedge i}\wedge E^c(n)\wedge\gmp))\\
       \xrightarrow{u}\uhom(\bb G_m^{\wedge n+1},C_*Fr(\bb G_m^{\wedge i}\wedge E^c(n+1))).
    \end{gather*}

For any $n\geqslant 0$, construct
a map of bispectra $f_n\colon ev_{\bb G_m}(\bb GC_*Fr^{E}[n])\to R^n_{\gmp}C_*Fr(E^c)$
as the composition at each level $i\geqslant 0$
   \begin{multline*}
    f_{n,i}\colon \uhom(\gmpn,C_*Fr(\bb G_m^{\wedge i}\wedge E^c(n)))\xrightarrow{tw}
     \uhom(\gmpn,C_*Fr(E^c(n)\wedge\bb G_m^{\wedge i}))\xrightarrow{u}\\
     \to\uhom(\gmpn,C_*Fr(E^c(n+i))).
    \end{multline*}

\begin{lem}\label{FrEr}
Each map $f_n$, $n\geqslant 0$, is a morphism of bispectra commuting
with stabilization maps $ev_{\bb G_m}(\bb GC_*Fr^{E}[n])\to ev_{\bb G_m}(\bb GC_*Fr^{E}[n+1])$ and
$R^n_{\gmp}C_*Fr(E^c)\to R^{n+1}_{\gmp}C_*Fr(E^c)$. In particular, they
induce a map of bispectra
   $$f\colon ev_{\bb G_m}(\cc M_{fr}^{E})\to R^\infty_{\gmp}C_*Fr(E^c).$$
\end{lem}

\begin{proof}
The following diagram commutes:\footnotesize
\[\xymatrixcolsep{0.075in}
\xymatrix{
\uhom(\gmpn,C_*Fr(\bb G_m^{\wedge i}\wedge E^c(n)))\wedge\gmp\ar[r]^{tw}\ar[d] &
\uhom(\gmpn,C_*Fr(E^c(n)\wedge\bb G_m^{\wedge i}))\wedge\gmp\ar[r]^(0.52){u}\ar[d] & \uhom(\gmpn,C_*Fr(E(n+i)))\wedge\gmp\ar[d]\\
\uhom(\gmpn,C_*Fr(\bb G_m^{\wedge i+1}\wedge E^c(n)))\ar[r]^{tw} &\uhom(\gmpn,C_*Fr(E^c(n)\wedge\bb G_m^{\wedge i+1}))
\ar[r]^{u} &\uhom(\gmpn,C_*Fr(E({n+i+1}))),}\]
\normalsize
where the left vertical arrow is the $i$th bonding map of
$ev_{\bb G_m}(\bb GC_*Fr^{E}[n])$, and the right vertical map is the $i$th
bonding map of $R^n_{\gmp}C_*Fr(E^c)$. We see that each map $f_n$ is a
morphism of bispectra. Consider a commutative diagram\footnotesize
\[\xymatrixcolsep{0.08in}
\xymatrix{
\uhom(\gmpn,C_*Fr(\bb G_m^{\wedge i}\wedge E^c(n)))\ar[r]^{tw}\ar[d] &\uhom(\gmpn,C_*Fr(E^c(n)\wedge\bb G_m^{\wedge i}))\ar[r]^{u}\ar[d]
& \uhom(\gmpn,C_*Fr(E(n+i)))\ar[d]\\
\uhom(\bb G_m^{\wedge n+1},C_*Fr(\bb G_m^{\wedge i}\wedge E^c(n+1)))\ar[r]^{tw} &
\uhom(\bb G_m^{\wedge n+1},C_*Fr(E^c(n+1)\wedge\bb G_m^{\wedge i}))\ar[r]^(.52){u} &\uhom(\bb G_m^{\wedge n+1},C_*Fr(E(n+1+i))), }\]
\normalsize
in which the left vertical map is the stabilization
$ev_{\bb G_m}(\bb GC_*Fr^{E}[n])(i)\to ev_{\bb G_m}(\bb GC_*Fr^{E}[n+1])(i)$
and the right vertical map is the stabilization map
$R^n_{\gmp}C_*Fr(E^c)(i)\to R^{n+1}_{\gmp}C_*Fr(E^c)(i)$. The middle vertical arrow equals
the composite map
   \begin{gather*}
     \uhom(\gmpn,C_*Fr(E^c(n)\wedge\bb G_m^{\wedge i}))\to\uhom(\bb G_m^{\wedge n+1},C_*Fr(E^c(n)\wedge\bb G_m^{\wedge i+1}))\\
     \xrightarrow{(\chi_{i,1})_*}\uhom(\bb G_m^{\wedge n+1},C_*Fr(E^c(n)\wedge\bb G_m^{\wedge 1+i}))\xrightarrow{u}
     \uhom(\bb G_m^{\wedge n+1},C_*Fr(E^c(n+1)\wedge \bb G_m^{\wedge i})),
   \end{gather*}
where $(\chi_{i,1})_*$ is induced by the shuffle map
$\chi_{i,1}:\bb G_m^{\wedge i+1}\to\bb G_m^{\wedge 1+i}$. For commutativity of the right
square we also use the fact that the diagram
\[\xymatrix{E(n)\wedge\bb G_m^{\wedge i+1}\ar[r]^{u}\ar[d]_{\id\wedge\chi_{i,1}}&E_{n+i}\wedge\gmp\ar[r]^{u}&E(n+i+1)\ar[d]^{1\oplus\chi_{i,1}}\\
E(n)\wedge \bb G_m^{\wedge 1+i}\ar[r]^{u}&E(n+1)\wedge\bb G_m^{\wedge i}\ar[r]^{u}&E(n+1+i)}\]
is commutative, because the compositions of horizontal maps are
$\Sigma_n\times\Sigma_{i+1}$-equivariant maps. Thus the maps
$f_{n,i}$ are compatible with stabilization.
\end{proof}

\begin{lem}\label{frrr}
The map $f$ induces a stable local equivalence of $S^1$-spectra for any $i\geqslant 0$:
\[f_i\colon ev_{\bb G_m}(\cc M_{fr}^{E})(i)\to R^\infty_{\gmp}C_*Fr(E^c)(i).\]
\end{lem}

\begin{proof}
The map $f_{n,i}\colon \uhom(\gmpn,C_*Fr(\bb G_m^{\wedge i}\wedge E^c(n)))
\to\uhom(\gmpn,C_*Fr(E^c(n+i)))$ fits into the following commutative diagram
\[\xymatrix{
\uhom(\gmpn,C_*Fr(\bb G_m^{\wedge i}\wedge E^c(n)))\ar@{=}[d]\ar[rr]^{f_{n,i}} && \uhom(\gmpn,C_*Fr(E^c(n+i)))\\
\Theta^{n}_{\gmp}(C_*Fr(\bb G_m^{\wedge i}\wedge E^c))(0)\ar[r]^{\upsilon} &
\Theta^{n}_{\gmp}(C_*Fr(E^c[i]))(0)\ar@{=}[r]&\Theta^{n}_{\gmp}(C_*Fr(E^c))(i)\ar[u]^{\chi_{i,n}}}\]
where $\chi_{i,n}\in\Sigma_{i+n}$ is the shuffle map permuting the first $i$
elements with the last $n$ elements and
$\Theta^{n}_{\gmp}(C_*Fr(E^c))$ is the bispectrum
   $$\Theta^{n}_{\gmp}(C_*Fr(E^c))(i)=\uhom(\gmpn,C_*Fr(E^c(i+n)))$$
with bonding maps being the compositions
   \begin{multline*}
    \uhom(\gmpn,C_*Fr(E^c(i+n)))\to\uhom(\bb G_m^{\wedge n+1},C_*Fr(E^c(i+n)\wedge\gmp))\to\\
       \xrightarrow{u}\uhom(\bb G_m^{\wedge n+1},C_*Fr(E^c(i+n+1))).
    \end{multline*}
Note that the maps of the diagram are compatible with stabilization maps, and
hence we can pass to the colimit over $n$ arriving at the commutative diagram
\[\xymatrix{
ev_{\bb G_m}(\cc M_{fr}^{E})(i)\ar@{=}[d]\ar[rr]^{f_{i}} && R^\infty_{\gmp}C_*Fr(E^c)(i)\\
\Theta^{\infty}_{\gmp}(C_*Fr(\bb G_m^{\wedge i}\wedge E^c))(0)\ar[r]^{\upsilon} &
\Theta^{\infty}_{\gmp}(C_*Fr(E^c[i]))(0)\ar@{=}[r] &\Theta^{\infty}_{\gmp}(C_*Fr(E^c))(i)\ar[u]^{\cong}}\]

We see that $f_i$ is a stable local equivalence of $S^1$-spectra if and only
if $\upsilon$ is. The latter is the case by Lemmas~\ref{upsilon} and~\ref{abtheta}.
\end{proof}

By Lemma \ref{frrr} the arrow $f: ev_{\bb G_m}(\cc M_{fr}^E)\to R^\infty_{\gmp}C_*Fr(E^c)$
is a levelwise {stable} local equivalence. We have constructed
the desired framed spectral functor
$\cc M_{fr}^E$ and a zigzag of stable motivic equivalences
  $$E^c\to R^\infty_{\gmp}C_*Fr(E^c)\xleftarrow{} ev_{\bb G_m}(\cc M_{fr}^E).$$
Thus $E$ is isomorphic to $ev_{\bb G_m}(\cc M_{fr}^E)$ in
$SH(k)$. This isomorphism is plainly functorial in $E$.

We claim that $\cc M_{fr}^E$ is a framed spectral functor. Clearly, all
simplicial enriched functors forming the spectral functor $\cc M_{fr}^E$ are framed
in the sense of Definition~\ref{frsp}.
Given $U\in Sm/k$ let
$\cc M_{fr}^{E}(-\times U)$ denote the spectral functor
   $$X\in\cc Fr_0(k)\mapsto\colim_n\uhom(\gmpn,C_*Fr(X_+\wedge U_+\wedge E^c(n))).$$

We already know that
$ev_{\bb G_m}(\cc M_{fr}^E(-\times U))\in SH^{fr}_{\nis}(k)$. The
above proof shows that
   $$U_+\wedge E^c\to R^\infty_{\gmp}C_*Fr(U_+\wedge E^c)\xleftarrow{} ev_{\bb G_m}(\cc M_{fr}^E(-\times U))$$
is a zigzag of stable motivic equivalences, functorial in $U$ and $E$.

Given an elementary Nisnevich square
   $$\xymatrix{U'\ar[d]\ar[r]&Y'\ar[d]\\
                       U\ar[r]&Y}$$
the square
   $$\xymatrix{U'_+\wedge E^c\ar[d]\ar[r]&Y'_+\wedge E^c\ar[d]\\
                       U_+\wedge E^c\ar[r]&Y_+\wedge E^c}$$
is homotopy cartesian in the stable motivic model structure of bispectra, and hence so is the square of motivically
fibrant bispectra
   $$\xymatrix{ev_{\bb G_m}(\cc M_{fr}^E(-\times U'))_f\ar[d]\ar[r]&ev_{\bb G_m}(\cc M_{fr}^E(-\times Y'))_f\ar[d]\\
                       ev_{\bb G_m}(\cc M_{fr}^E(-\times U))_f\ar[r]&ev_{\bb G_m}(\cc M_{fr}^E(-\times Y))_f}$$
(we use here Lemma~\ref{lemst}). It follows that the square of $S^1$-spectra
   $$\xymatrix{ev_{\bb G_m}(\cc M_{fr}^E(-\times U'))_f(0)\ar[d]\ar[r]&ev_{\bb G_m}(\cc M_{fr}^E(-\times Y'))_f(0)\ar[d]\\
                       ev_{\bb G_m}(\cc M_{fr}^E(-\times U))_f(0)\ar[r]&ev_{\bb G_m}(\cc M_{fr}^E(-\times Y))_f(0)}$$
is sectionwise homotopy cartesian in the stable model structure of ordinary $S^1$-spectra.
Thus $\cc M_{fr}^E$ is Nisnevich excisive (obviously, $\cc M_{fr}^E(\emptyset)=*$).

For the same reasons $pr:\cc M_{fr}^E(\bb A^1\times U)\to\cc M_{fr}^E(U)$ is a stable local
equivalence of $S^1$-spectra, and hence the claim.

So the functor $F\circ ev_{\bb G_m}:SH_{S^1}^{fr}[\Fr]\to SH(k)$ is an equivalence of categories. The compactly generated
triangulated category structure on $SH_{S^1}^{fr}[\Fr]$ is just the preimage of the structure on $SH(k)$.
Clearly, $\{\Omega_{\gmpn}\cc M_{fr}(X)\mid X\in Sm/k,n\geq 0\}$ are compact generators of
$SH_{S^1}^{fr}[\Fr]$. Since every $\cc X\in SH_{S^1}^{fr}[\Fr]$
is Nisnevich excisive, it follows that the canonical map of $S^1$-spectra $\cc X(U\sqcup V)\to\cc X(U)\times\cc X(V)$
is a stable equivalence for all $U,V\in Sm/k$. Therefore the canonical map of spectral functors
$\cc X\wedge S^1\to\cc X(-\otimes S^1)$ is pointwise a stable equivalence. It follows that the triangulated shift
$\cc X[1]$ is computed as $\cc X(-\otimes S^1)$ in $SH_{S^1}^{fr}[\Fr]$. This completes the proof of the theorem.
\end{proof}

\section{The triangulated category of framed motives}

\begin{dfn}
(1) We say that a framed spectral functor $\cc X\in SH_{S^1}^{fr}[\Fr]$ is {\it effective\/}
if for all $n>0$, $U\in Sm/k$ and a finitely generated field $K/k$
the ordinary $S^1$-spectrum $\cc X(\gmpn\times U)(\wh{\Delta}^\bullet_K)$
is stably trivial. In other words, $\cc X\in SH_{S^1}^{fr}[\Fr]$ is effective if and only if
its evaluation bispectrum $ev_{\bb G_m}(\cc X(-\times U))=(\cc X(U),\cc X(\gmp\times U),\ldots)$
is effective in the sense of Definition~\ref{dfneff}.

(2) The full subcategory of $SH_{S^1}^{fr}[\Fr]$ consisting of the effective spectral
functors is called the {\it category of framed motives\/} and denoted by
${\cc {SH}}^{fr}(k)$.
\end{dfn}

\begin{thm}\label{spfuneff}
The category of framed motives ${\cc {SH}}^{fr}(k)$ is
compactly generated triangulated with compact generators
$\{\cc M_{fr}(X)\mid X\in Sm/k\}$ and the shift functor $\cc X[1]=\cc X(-\otimes S^1)$.
Moreover, restriction of the composite functor
$F\circ ev_{\bb G_m}:SH_{S^1}^{fr}[\Fr]\to SH(k)$
to ${\cc {SH}}^{fr}(k)$ lands in $SH^{eff}(k)$ and
is an equivalence of triangulated categories $G:{\cc {SH}}^{fr}(k)\lra{\sim}SH^{eff}(k)$, where
$F:SH_{\nis}^{fr}(k)\to SH(k)$ is the triangulated equivalence of
Theorem~\ref{recover}.
\end{thm}

\begin{proof}
Given $\cc X\in{\cc {SH}}^{fr}(k)$, the evaluation bispectrum $ev_{\bb G_m}(\cc X)$ is
effective in the sense of Definition~\ref{dfneff}. By Theorem~\ref{thmeff} $ev_{\bb G_m}(\cc X)$
belongs to $SH^{eff}(k)$. It follows that restriction of the composite functor
$F\circ ev_{\bb G_m}:SH_{S^1}^{fr}[\Fr]\to SH(k)$
to ${\cc {SH}}^{fr}(k)$ lands in $SH^{eff}(k)$. We arrive therefore at a functor
$G:{\cc {SH}}^{fr}(k)\to SH^{eff}(k)$, which is fully faithful by construction.
Let $E\in SH^{eff}(k)$ be a symmetric fibrant bispectrum. By the proof of Theorem~\ref{spfuncat}
there is a framed spectral functor $\cc M^E_{fr}$ such that $E$ is isomorphic to
$ev_{\bb G}(\cc M^E_{fr})$ in $SH(k)$. Theorem~\ref{thmeff} implies $\cc M^E_{fr}$ is effective.
Since every bispectrum is isomorphic in $SH(k)$ to an associated symmetric fibrant
bispectrum, we see that $G$ is an equivalence of categories.

Next, the compactly generated triangulated category structure on ${\cc {SH}}^{fr}(k)$
is inherited from $SH_{\nis}^{fr}(k)$. It is also the preimage of the same structure on $SH^{eff}(k)$.
The rest now follows from Theorem~\ref{spfuncat}.
\end{proof}

In order to formulate the next theorem, we need to define several functors first. Let
   $$H:SH^{eff}(k)\to{\cc {SH}}^{fr}(k)$$
be the functor that takes a bispectrum $E$ to the spectral framed functor
$\cc M_{fr}^{E_{f}}$, where $E_{f}$ is the symmetric fibrant bispectrum
associated to $E$ (we can always do this functorially on the level of bispectra).
The proof of the preceding theorem shows that $\cc M_{fr}^{E_{f}}$ is effective
if and only if $E$ is and that $H$ is quasi-inverse to $G:{\cc {SH}}^{fr}(k)\to SH^{fr}(k)$.

Next, define a triangulated functor
   $$\cc M_{fr}:SH_{S^1}(k)\to{\cc {SH}}^{fr}(k)$$
as follows. It takes an $S^1$-spectrum $A$ to the spectral functor
   $$X\in\Fr\mapsto C_*Fr(X_+\wedge A^c),$$
where $A^c$ is the (functorial) cofibrant replacement of $A$ in the stable
projective model structure of $S^1$-spectra. Note that $A^c$ is levelwise a
sequential colimit of simplicial smooth schemes. It follows from Lemma~\ref{yyy}
that $\cc M_{fr}(A)$ is a framed spectral functor. Given a stable motivic equivalence
of $S^1$-spectra $f:A\to B$, consider a commutative diagram of bispectra
   $$\xymatrix{\Sigma^\infty_{\gmp}A^c\ar[r]\ar[d]_f&C_*Fr(\Sigma^\infty_{\gmp}A^c)\ar[r]^(.5)\cong\ar[d]_f&ev_{\bb G}(\cc M_{fr}(A))\ar[d]^{f}\\
                         \Sigma^\infty_{\gmp}B^c\ar[r]&C_*Fr(\Sigma^\infty_{\gmp}B^c)\ar[r]^(.5)\cong&ev_{\bb G}(\cc M_{fr}(B)).}$$
Here the right horizontal arrows are isomorphisms given by swappings $\gmpn\wedge-\cong-\wedge\gmpn$.
By Theorem~\ref{Thm_A_C_*Fr_A} the left horizontal maps are isomorphisms in
$SH(k)$, and hence so is the right vertical map. So $\cc M_{fr}:SH_{S^1}(k)\to{\cc {SH}}^{fr}(k)$ is indeed a functor.
Clearly, it is triangulated.

Finally using Lemma~\ref{lemst}, we define a triangulated functor
   $$U:{\cc {SH}}^{fr}(k)\to SH_{S^1}(k)$$
(``$U$'' for underlying) taking $\cc X\in{\cc {SH}}^{fr}(k)$ to the stable local replacement $\cc X(pt)_f$
of the underlying $S^1$-spectrum $\cc X(pt)$.

We now document the above constructions as follows.

\begin{thm}\label{ewq}
In the diagram of triangulated functors
   $$\xymatrix{SH_{S^1}(k)\ar@/^/[rr]^{\Sigma^\infty_{\gmp}}\ar@/^/[dr]^(.6){\cc M_{fr}}&&SH^{eff}(k)\ar@/^/[ll]^(.45){\Omega^\infty_{\gmp}}\ar@/_/[dl]_(.55)H\\
                        &{\cc {SH}}^{fr}(k)\ar@/_/[ur]_G\ar@/^/[ul]^U}$$
all three pairs of functors are adjoint pairs, $G\circ\cc M_{fr}$ is equivalent to $\Sigma^\infty_{\gmp}$,
$U\circ H$ is equivalent to $\Omega^\infty_{\gmp}$.
\end{thm}

We finish the paper with a $\gmp$-connectedness result for
effective motivic framed spectral functors. Recall that for every $n\geq 0$
one defines the category $f_n(SH_{S^1}(k))$ as a full triangulated subcategory of
$SH_{S^1}(k)$ compactly generated by the suspension spectra $\Sigma^\infty_{S^1}X_+\wedge\gmpn$,
$X\in Sm/k$. We say that a $S^1$-spectrum $A\in SH_{S^1}(k)$ is {\it $n$-connected
in the $\gmp$-direction}, $n\geq 0$, if $A\in f_n(SH_{S^1}(k))$.

\begin{cor}\label{gmpconn}
For any effective motivic framed spectral functor $\cc X\in{\cc {SH}}^{fr}(k)$ and any $n\geq 0$,
the $S^1$-spectrum $\cc X(\gmpn)$ is $n$-connected in the $\gmp$-direction.
\end{cor}

\begin{proof}
This follows from Theorem~\ref{ewq} and~\cite[7.5.1]{Lev}.
\end{proof}

\end{document}